\documentclass[a4paper,
               11pt,
               pdftex,
               normalheadings,
               headsepline,
               footsepline,
               onecolumn,
               headinclude,
               footinclude,
               DIV14,
               abstracton]
{scrartcl}

\usepackage
{
    graphicx,
    amssymb,
    amsmath,
    amsthm,
    xcolor,
    dsfont,
    algpseudocode,
    authblk,
}

\usepackage[bf]{caption}
\usepackage[colorlinks,
            pdffitwindow=false,
            plainpages=false,
            pdfpagelabels=true,
            pdfpagemode=UseOutlines,
            pdfpagelayout=SinglePage,
            bookmarks=false,
            colorlinks=true,
            hyperfootnotes=false,
            linkcolor=blue,
            citecolor=green!50!black]
{hyperref}

\usepackage{enumerate}
\usepackage{algorithm}

\usepackage{hyphenat}
\DeclareMathAlphabet{\mathpzc}{OT1}{pzc}{m}{it}

\newcommand{\R}{\mathbb{R}}

\newtheorem{theorem}{Theorem}[section]
\newtheorem{corollary}[theorem]{Corollary}
\newtheorem{lemma}[theorem]{Lemma}

\newtheorem{definition}[theorem]{Definition}
\newtheorem{remark}[theorem]{Remark}

\makeatletter
\renewcommand*\env@matrix[1][*\c@MaxMatrixCols c]{%
  \hskip -\arraycolsep
  \let\@ifnextchar\new@ifnextchar
  \array{#1}}
\makeatother

\newcommand{\PS}{\mathcal{P}_{{S}}}
\newcommand{\PSsub}{\mathcal{P}_{{S},sub}}
\newcommand{\PF}{\mathcal{P}_{{F}}}
\newcommand{\PSeps}{\mathcal{P}_{{S},\epsilon}}
\newcommand{\PSxi}{\mathcal{P}_{{S},\xi}}
\newcommand{\PSepsxi}{\mathcal{P}_{{S},\epsilon,\xi}}

\newcommand{\qr}{q_u}

\newcommand{\ftilde}{\widetilde{f}}

\newcommand{\alphahat}{\widehat{\alpha}}

\newcommand{\SetQ}{\mathcal{Q}}
\newcommand{\SetQr}{{\mathcal{Q}}_u}

\begin{document}

% document title and author
%\title{Multiobjective Optimization with Inexact Gradients}
\title{Gradient-Based Multiobjective Optimization with Uncertainties}
\author[*]{Sebastian Peitz}
\author[*]{Michael Dellnitz}
\affil[*]{\normalsize Department of Mathematics, Paderborn University, Warburger Str.~100, D-33098 Paderborn}
%\date{\normalsize \vspace{0.2cm} \blue{\today \ \ (\thistime)}}

\maketitle

%%%%%%%%%%%%%%%%%%%%%%%%%%%%%%%%%%%%%%%%%%%%%%%%%%%%%%%%%%%%%%%%%%%%%%%%%%%%%%%%%%%%%%%%%%%%%%
%% Abstract
%%%%%%%%%%%%%%%%%%%%%%%%%%%%%%%%%%%%%%%%%%%%%%%%%%%%%%%%%%%%%%%%%%%%%%%%%%%%%%%%%%%%%%%%%%%%%%
\begin{abstract}
In this article we develop a gradient-based algorithm for the solution of multiobjective optimization problems with uncertainties. To this end, an additional condition is derived for the descent direction in order to account for inaccuracies in the gradients and then incorporated in a subdivison algorithm for the computation of global solutions to multiobjective optimization problems. Convergence to a superset of the Pareto set is proved and an upper bound for the maximal distance to the set of substationary points is given. Besides the applicability to problems with uncertainties, the algorithm is developed with the intention to use it in combination with model order reduction techniques in order to efficiently solve PDE-constrained multiobjective optimization problems.
\end{abstract}

%%%%%%%%%%%%%%%%%%%%%%%%%%%%%%%%%%%%%%%%%%%%%%%%%%%%%%%%%%%%%%%%%%%%%%%%%%%%%%%%%%%%%%%%%%%%%%
%% Introduction
%%%%%%%%%%%%%%%%%%%%%%%%%%%%%%%%%%%%%%%%%%%%%%%%%%%%%%%%%%%%%%%%%%%%%%%%%%%%%%%%%%%%%%%%%%%%%%
\section{Introduction}
\label{sec:Introduction}
In many applications from industry and economy, one is interested in simultaneously optimizing several criteria. For example, in transportation one wants to reach a destination as fast as possible while minimizing the energy consumption. This example illustrates that in general, the different objectives contradict each other. Therefore, the task of computing the set of optimal compromises between the conflicting objectives, the so-called \emph{Pareto set}, arises. This leads to a multiobjective optimization problem (MOP). Based on the knowledge of the Pareto set, a \emph{decision maker} can use this information either for improved system design or for changing parameters during operation, as a reaction on external influences or changes in the system state itself.

Multiobjective optimization is an active area of research. Different methods exist to address MOPs, e.g.~deterministic approaches \cite{Mie99,Ehr05}, where ideas from scalar optimization theory are extended to the multiobjective situation. In many cases, the resulting solution method involves solving multiple scalar optimization problems consecutively. Continuation methods make use of the fact that under certain smoothness assumptions, the Pareto set is a manifold that can be approximated by continuation methods known from dynamical systems theory \cite{Hil01}. Another prominent approach is based on evolutionary algorithms \cite{CCvVL02}, where the underlying idea is to evolve an entire set of solutions (population) during the optimization process. Set oriented methods provide an alternative deterministic approach to the solution of MOPs. Utilizing subdivision techniques (cf.~\cite{DSH05,Jah06,SWO+13}), the desired Pareto set is approximated by a nested sequence of increasingly refined box coverings. In the latter approach, gradients are evaluated to determine a descent direction for all objectives.

Many solution approaches are gradient-free since the benefit of derivatives is less established than in single objective optimization \cite{Bos12}. Exceptions are scalarization methods since the algorithms used in this context mainly stem from scalar optimization. Moreover, in so-called \emph{memetic algorithms} \cite{NCM12}, evolutionary algorithms are combined with local search strategies, where gradients are utilized (see e.g.~\cite{LSCCS10,SML+16,SASL16}). Finally, several authors also develop gradient-based methods for MOPs directly. In \cite{FS00,SSW02,Des12}, algorithms are developed where a single descent direction is computed in which all objectives decrease. In \cite{Bos12}, a method is presented by which the entire set of descent directions can be determined, an extension of Newton's method to MOPs with quadratic convergence is presented in \cite{FDS09}. %\cite{Des12} \cite{HSK06}

Many real world problems possess uncertainties for various reasons such as the ignorance of the exact underlying dynamical system or unknown material properties. The use of reduced order models in order to decrease the computational effort also introduces errors which can often be quantified \cite{TV09}. A similar concept exists in the context of evolutionary computation, where \emph{Surrogate-assisted evolutionary computation} (see e.g.~\cite{Jin11} for a survey) is often applied in situations where the exact model is too costly to solve.
In this setting, the set of \emph{almost Pareto optimal points} has to be computed (see also \cite{Whi86}, where the idea of \emph{$\epsilon$ efficiency} was introduced). 
Many researchers are investigating problems related to uncertainty quantification and several authors have addressed multiobjective optimization problems with uncertainties. In \cite{Hug01} and in \cite{Tei01}, probabilistic approaches to multiobjective optimization problems with uncertainties were derived independently. In \cite{DG05,DMM05,BZ06,SM08}, evolutionary algorithms were developed for problems with uncertain or noisy data in order to compute robust approximations of the Pareto set. In \cite{SCCTT08}, stochastic search was used, cell mapping techniques were applied in \cite{HSS13} and in \cite{EW07}, the weighted sum method was extended to uncertainties. %However, all these works have in common that no gradient information was incorporated in the numerical methods.

In this article, we present extensions to the gradient-free and gradient-based global subdivision algorithms for unconstrained MOPs developed in \cite{DSH05} which take into account inexactness in both the function values and the gradients. The algorithms compute a set of solutions which is a superset of the Pareto set with an upper bound for the distance to the Pareto set. The remainder of the article is organized in the following manner. In Section~\ref{sec:MO}, we give a short introduction to multiobjective optimization in general and gradient-based descent directions for MOPs. In Section~\ref{sec:InexactGradients}, a descent direction for all objectives under inexact gradient information is developed and an upper bound for the distance to a Pareto optimal point is given. In Section~\ref{sec:Subdivision}, the subdivision algorithm presented in \cite{DSH05} is extended to inexact function and gradient values before we present our results in Section~\ref{sec:Results} and draw a conclusion in Section~\ref{sec:Conclusion}.

%%%%%%%%%%%%%%%%%%%%%%%%%%%%%%%%%%%%%%%%%%%%%%%%%%%%%%%%%%%%%%%%%%%%%%%%%%%%%%%%%%%%%%%%%%%%%%
%% Multiobjective Optimization
%%%%%%%%%%%%%%%%%%%%%%%%%%%%%%%%%%%%%%%%%%%%%%%%%%%%%%%%%%%%%%%%%%%%%%%%%%%%%%%%%%%%%%%%%%%%%%
\section{Multiobjective Optimization}
\label{sec:MO}
Consider the continuous, unconstrained multiobjective optimization problem
\begin{align}
	\min_{x\in\R^n} F(x) = \min_{x\in\R^n} \left( \begin{array}{c} f_1(x) \\ \vdots \\ f_k(x) \end{array} \right), \tag{MOP} \label{eq:MOP}
\end{align}
where $F: \R^n \rightarrow \R^k$ is a vector valued objective function with continuously differentiable objective functions $f_i: \R^n \rightarrow \R, i = 1, \ldots, k$.
The space of the parameters $x$ is called the \emph{decision space} and the function $F$ is a mapping to the $k$-dimensional \emph{objective space}. In contrast to single objective optimization problems, there exists no total order of the objective function values in $\mathbb{R}^k, k \geq 2$ (unless the objectives are not conflicting). Therefore, the comparison of values is defined in the following way \cite{Mie99}:
\begin{definition}
	Let $v, w \in \mathbb{R}^k$. The vector $v$ is \emph{less than} $w$ $(v <_p w)$, if $v_i < w_i$ for all $i \in \left\lbrace 1, \ldots, k \right\rbrace$. The relation $\leq_p$ is defined in an analogous way.
\end{definition}
A consequence of the lack of a total order is that we cannot expect to find isolated optimal points. Instead, the solution of \eqref{eq:MOP} is the set of optimal compromises, the so-called \emph{Pareto set} named after Vilfredo Pareto:
\begin{definition}~
	\label{def:Pareto_optimality}
	\begin{itemize}%\begin{enumerate}[(a)]
		\item[(a)] A point $x^* \in \R^n$ \emph{dominates} a point $x \in \R^n$, if $F(x^*) \leq_p F(x)$ and $F(x^*) \neq F(x)$.
		\item[(b)] A point $x^* \in \R^n$ is called \emph{(globally) Pareto optimal} if there exists no point $x \in \R^n$ dominating $x^*$. The image $F(x^*)$ of a (globally) Pareto optimal point $x^*$ is called a \emph{(globally) Pareto optimal value}.
		\item[(c)] The set of non-dominated points is called the \emph{Pareto set} $\PS$, its image the \emph{Pareto front} $\PF$.
	\end{itemize}
\end{definition}
Consequently, for each solution that is contained in the Pareto set, one can only improve one objective by accepting a trade-off in at least one other objective. That is, roughly speaking, in a two-dimensional problem, we are interested in finding the ``lower left'' boundary of the reachable set in objective space (cf.~Figure~\ref{fig:MOP_Example} (b)). A more detailed introduction to multiobjective optimization can be found in e.g.~\cite{Mie99,Ehr05}.
\begin{figure}[h!]
	\centering
	\parbox[b]{0.49\textwidth}{\centering \includegraphics[width=0.45\textwidth]{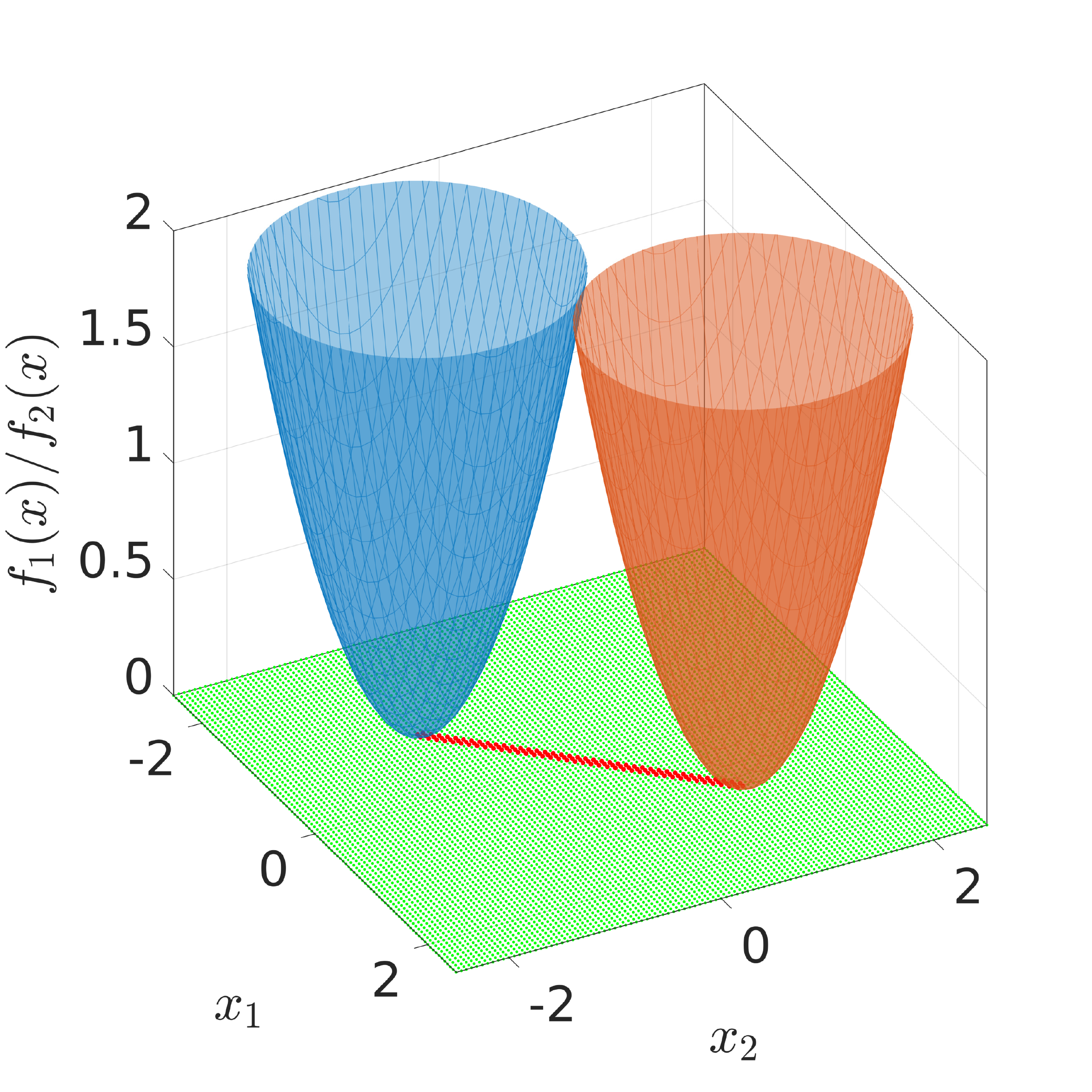}\\(a)}
	\parbox[b]{0.49\textwidth}{\centering \includegraphics[width=0.4\textwidth]{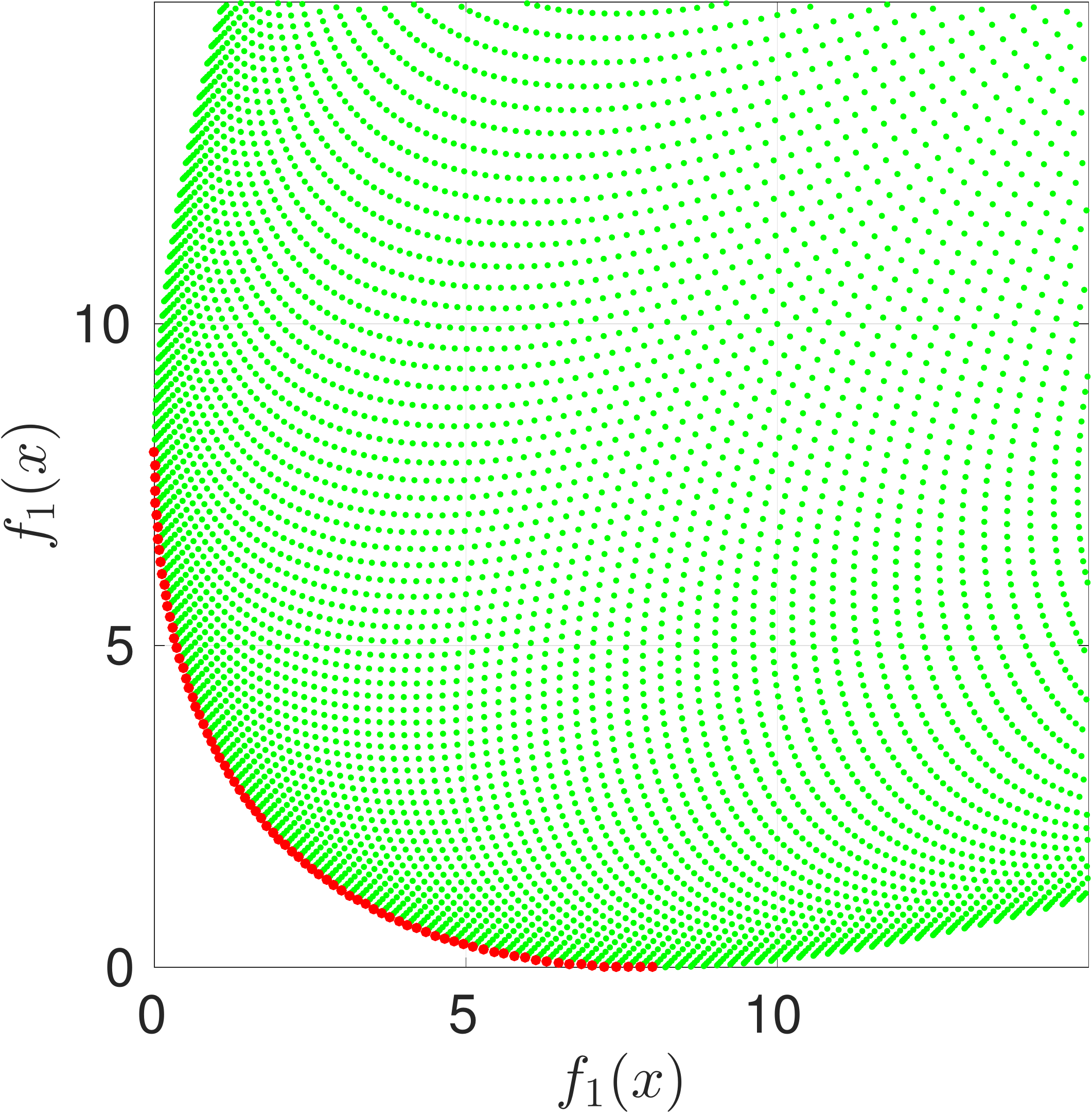}\\(b)}
	\caption{The red lines depict the Pareto set (a) and Pareto front (b) of an exemplary multiobjective optimization problem (two paraboloids) of the form \eqref{eq:MOP} with $n = 2$ and $k = 2$.}
	\label{fig:MOP_Example}
\end{figure}

Similar to single objective optimization, a necessary condition for optimality is based on the gradients of the objective functions. In the multiobjective situation, the corresponding Karush-Kuhn-Tucker (KKT) condition is as follows:
\begin{theorem}[\cite{KT51}]
	Let $x^*$ be a Pareto point of \eqref{eq:MOP}. Then, there exist nonnegative scalars $\alpha_1, \ldots, \alpha_k \ge 0$ such that
	\begin{align}
	\sum_{i=1}^k \alpha_i = 1 \text{ and } \sum_{i=1}^k \alpha_i \nabla f_i(x^*) = 0. \label{eq:MOP_optimality}
	\end{align}
\end{theorem}

\noindent Observe that \eqref{eq:MOP_optimality} is only a necessary condition for a point $x^*$ 
to be a Pareto point and the set of points satisfying \eqref{eq:MOP_optimality} is called the set of {\em substationary points}
$\PSsub$. Under additional smoothness assumptions (see \cite{Hil01}) $\PSsub$ is
locally a $k-1$-dimensional manifold. Obviously $\PSsub$ is a superset of the Pareto set $\PS$.

If $x\not\in \PSsub$ then the KKT conditions can be utilized in order to identify a descent direction $q(x)$ for which all objectives are non-increasing, i.e.:
\begin{align}
	-\nabla f_i(x) \cdot q(x) \ge 0,\quad i=1, \ldots, k. \label{eq:condition_descent_all}
\end{align}
One way to compute a descent direction satisfying \eqref{eq:condition_descent_all} is to solve the following auxiliary optimization problem \cite{SSW02}:
\begin{align}
	\min_{\alpha \in \R^k} \left\lbrace \Big\| \sum_{i=1}^k \alpha_i \nabla f_i(x) \Big\|_2^2 \quad \Big| ~ \alpha_i \ge 0,~i=1, \ldots, k, \sum_{i=1}^k \alpha_i = 1 \right\rbrace. \tag{QOP} \label{eq:QOP}
\end{align}
Using \eqref{eq:QOP}, we obtain the following result:
\begin{theorem}[\cite{SSW02}] \label{th:SSW}
	Define $q: \R^n \rightarrow \R^n$ by
	\begin{align}
		q(x) = -\sum_{i=1}^k \alphahat_i \nabla f_i(x), \label{eq:q}
	\end{align}
	where $\alphahat$ is a solution of \eqref{eq:QOP}. Then either $q(x) = 0$ and $x$ satisfies \eqref{eq:MOP_optimality}, or $q(x)$ is a descent direction for all objectives $f_1(x), \ldots, f_k(x)$ in $x$. Moreover,
	$q(x)$ is locally Lipschitz continuous.
\end{theorem}
\begin{remark}
As in the classical case of scalar optimization there exist in general infinitely many valid descent directions
$q(x)$ and using the result from Theorem~\ref{th:SSW} yields one particular direction. As already stated in the introduction, there are alternative ways to compute such a direction, see e.g.~\cite{FS00} for the computation of a single direction or \cite{Bos12}, where the entire set of descent directions is determined.
\end{remark}

%%%%%%%%%%%%%%%%%%%%%%%%%%%%%%%%%%%%%%%%%%%%%%%%%%%%%%%%%%%%%%%%%%%%%%%%%%%%%%%%%%%%%%%%%%%%%%
%% Inexact Gradients
%%%%%%%%%%%%%%%%%%%%%%%%%%%%%%%%%%%%%%%%%%%%%%%%%%%%%%%%%%%%%%%%%%%%%%%%%%%%%%%%%%%%%%%%%%%%%%
\section{Multiobjective Optimization with Inexact Gradients}
\label{sec:InexactGradients}
Suppose now that we only have approximations $\ftilde_i(x), \nabla \ftilde_i(x)$ of the objectives $f_i(x)$ and their gradients $\nabla f_i(x)$, $i = 1, \ldots, k$, respectively. To be more precise we assume that
\begin{align}
	\ftilde_i(x) &= f_i(x) + \bar{\xi}_i,\quad &\|\ftilde_i(x) - f_i(x)\|_2 = \|\bar{\xi}_i\|_2 \le \xi_i, \label{eq:function_error}\\
	\nabla \ftilde_i(x) &= \nabla f_i(x) + \bar{\epsilon}_i, \quad &\|\nabla \ftilde_i(x) - \nabla f_i(x)\|_2 = \|\bar{\epsilon}_i\|_2 \le \epsilon_i, \label{eq:gradient_error}
\end{align}
where the upper bounds $\xi_i, \epsilon_i$ are given. %For technical purposes we have to assume that
In the following, when computing descent directions we will assume that
\[
\epsilon_i \le \|\nabla f_i(x)\|_2 \quad \mbox{for all $x$,}
\]
since otherwise, we are already in the vicinity of the set of stationary points as will be shown in Lemma~\ref{lemma:norm_q_tilde}. 
Using \eqref{eq:gradient_error} we can derive an upper bound for the angle between the exact and the inexact gradient by elementary geometrical considerations (cf. Figure~\ref{fig:angle_gradient} (a)):
\begin{align}
	\sphericalangle(\nabla f_i(x), \nabla \ftilde_i(x)) = \arcsin\left( \frac{\|\bar{\epsilon}_i\|_2}{\|\nabla f_i(x)\|_2} \right) \leq \arcsin\left( \frac{\epsilon_i}{\|\nabla f_i(x)\|_2} \right) =: \varphi_i. \label{eq:def_phi}
\end{align}
Here we denote the largest possible angle (i.e.~the ``worst case'') by $\varphi_i$. Based on this angle, one can easily define a condition for the inexact gradients such that the exact gradients could satisfy \eqref{eq:MOP_optimality} for the first time. This is precisely the case if each inexact gradient deviates from the hyperplane defined by \eqref{eq:MOP_optimality} at most by $\varphi_i$, see Figure~\ref{fig:angle_gradient} (b). This motivates the definition of an \emph{inexact descent direction}:
\begin{definition}\label{def:InexactDescentDirection}
	A direction analog to \eqref{eq:q} but based on inexact gradients, i.e.
	\begin{align}
		\qr(x) = -\sum_{i=1}^k \alphahat_i \nabla \ftilde_i(x) \quad \mbox{with $\alphahat_i \ge 0$ for $i=1, \ldots, k$ and $\sum_{i=1}^k \alphahat_i = 1$}, \label{eq:q_tilde}
	\end{align}
	is called \emph{inexact descent direction}.
\end{definition}
We can prove an upper bound for the norm of $\qr(x^*)$ when $x^*$ satisfies the KKT condition of the exact problem:
\begin{lemma}\label{lemma:norm_q_tilde}
	Consider the multiobjective optimization problem \eqref{eq:MOP} with inexact gradient information according to \eqref{eq:gradient_error}. 
	Let $x^*$ be a point satisfying the KKT conditions \eqref{eq:MOP_optimality} for the exact problem. Then, the inexact descent direction $\qr(x^*)$ is bounded by
	\[
		\|\qr(x^*)\|_2 \le \| \epsilon \|_{\infty}. %\le \sum_{i=1}^{k} \epsilon_i.
	\]
\end{lemma}
\begin{proof}
	Since $x^*$ satisfies \eqref{eq:MOP_optimality}, we have $\sum_{i=1}^k\alphahat_i \nabla f_i(x^*) = 0$. Consequently,
	\[
		\sum_{i=1}^k\alphahat_i \nabla \ftilde_i(x) = \sum_{i=1}^k \alphahat_i \left(\nabla f_i(x) + \bar{\epsilon}_i\right) = \sum_{i=1}^k \alphahat_i \bar{\epsilon}_i,
	\]
	and thus,
	\begin{align*}
		\Big\|\sum_{i=1}^k\alphahat_i \nabla \ftilde_i(x) \Big\|_2 &= \Big\|\sum_{i=1}^k\alphahat_i \bar{\epsilon}_i \Big\|_2 \le \sum_{i=1}^k\alphahat_i \|\bar{\epsilon}_i \|_2 \le \sum_{i=1}^k\alphahat_i \epsilon_i \\
		&\le \sum_{i=1}^k\alphahat_i \| \epsilon \|_{\infty} = \| \epsilon \|_{\infty} \sum_{i=1}^k\alphahat_i = \| \epsilon \|_{\infty}.
	\end{align*}
\end{proof}

\begin{figure}[htb]
	\centering
	\parbox[b]{0.6\textwidth}{\centering \includegraphics[width=0.6\textwidth]{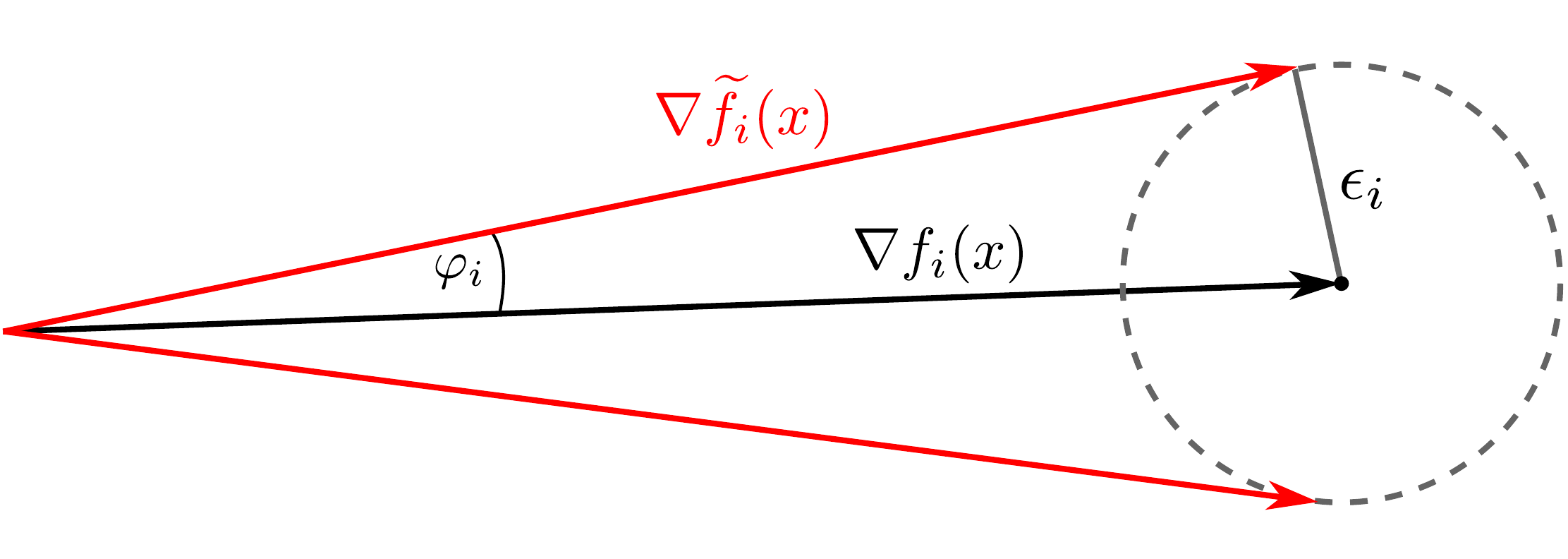}\\(a)}\\~\\~\\
	\parbox[b]{0.6\textwidth}{\centering \includegraphics[width=0.6\textwidth]{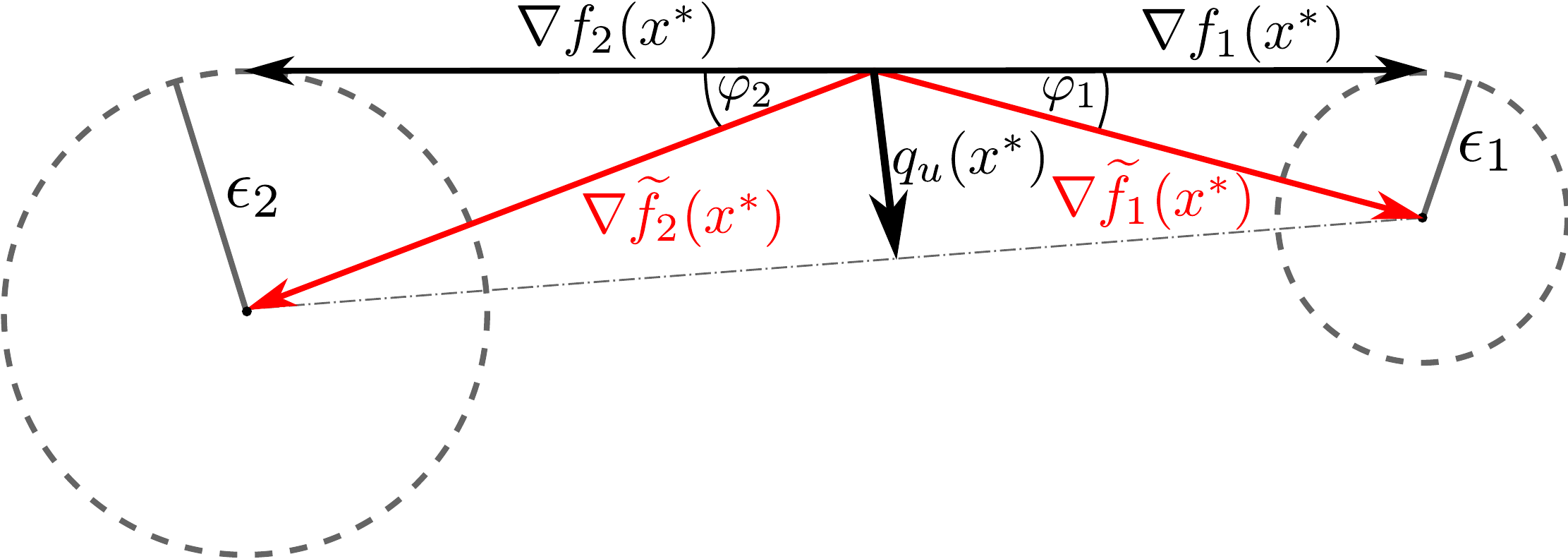}\\(b)}
	\caption{(a) Maximal angle between the exact and the inexact gradient in dependence on the error $\epsilon_i$. (b) Maximal angle between the inexact gradients in the situation where $x^*$ satisfies \eqref{eq:MOP_optimality} ($\pi - (\varphi_1 + \varphi_2)$ in the 2D case).}
	\label{fig:angle_gradient}
\end{figure}

A consequence of Lemma~\ref{lemma:norm_q_tilde} is that in the presence of inexactness, we cannot compute the set of points satisfying \eqref{eq:MOP_optimality} exactly. At best, we can compute the set of points determined by $\|\qr(x)\|_2 \le \| \epsilon \|_{\infty}$. In the following section we will derive a criterion for the inexact descent direction $\qr(x)$ which guarantees that it is also a descent direction for the exact problem when $x$ is sufficiently far away from the set of substationary points.

\subsection{Descent Directions in the Presence of Inexactness}
\label{subsec:InexactDescentDirection}
The set of valid descent directions for \eqref{eq:MOP} is a cone defined by the intersection of all half-spaces orthogonal to the gradients $\nabla f_1(x), \ldots, \nabla f_k(x)$ (cf.~Figure \ref{fig:angle_descent_set} (a)), i.e.~it consists of all directions $q(x)$ satisfying
\begin{align}
	\sphericalangle(q(x), -\nabla f_i(x)) \leq \frac{\pi}{2},\quad i = 1, \ldots, k. \label{eq:descent_cone_exact}
\end{align}
This fact is well known from scalar optimization theory, see e.g.~\cite{NW06}. Observe that here we allow directions
also to be valid for which all objectives are at least non-decreasing. 
In terms of the angle $\gamma_i \in [0, \pi/2]$ between the descent direction $q(x)$ and the hyperplane orthogonal to the gradient of the $i^{\mathsf{th}}$ objective, this can be expressed as
\begin{align}
	\gamma_i = \frac{\pi}{2} - \arccos\left( \frac{q(x) \cdot \left(-\nabla f_i(x)\right)}{\|q(x)\|_2 \cdot \|\nabla f_i(x)\|_2} \right) \geq 0,\quad i= 1, \ldots, k. \label{eq:def_gamma}
\end{align}
\begin{remark}\label{rem:gamma}
	Note that $\gamma_i$ is equal to $\pi/2$ when $q(x) = \nabla f_i(x)$ and approaches zero when a point $x$ approaches the set of substationary points, i.e.
	\[
		\lim_{\|q(x)\|_2 \rightarrow 0} \gamma_i \rightarrow 0 \quad \mbox{for $i = 1, \ldots, k$}.
	\]
\end{remark}

We call the set of all descent directions satisfying \eqref{eq:descent_cone_exact} the \emph{exact cone} $\SetQ$. If we consider inexact gradients according to \eqref{eq:gradient_error} then $\SetQ$ is reduced to the \emph{inexact cone} $\SetQr$ by the respective upper bounds for the angular deviation $\varphi_i$:
\begin{align}
	\gamma_i \geq \varphi_i \geq 0, \label{eq:gamma_i}\quad i= 1, \ldots, k.
\end{align}
This means that we require the angle between the descent direction $q(x)$ and the $i^{\mathsf{th}}$ hyperplane to be at least as large as the maximum deviation between the exact and the inexact gradient (cf. Figure~\ref{fig:angle_descent_set} (b)).
Thus, if an inexact descent direction $\qr(x)$ satisfies \eqref{eq:gamma_i}, then it is also a descent
direction for the exact problem. 
\begin{figure}[htb]
	\centering
	\parbox[b]{0.49\textwidth}{\centering \includegraphics[width=0.45\textwidth]{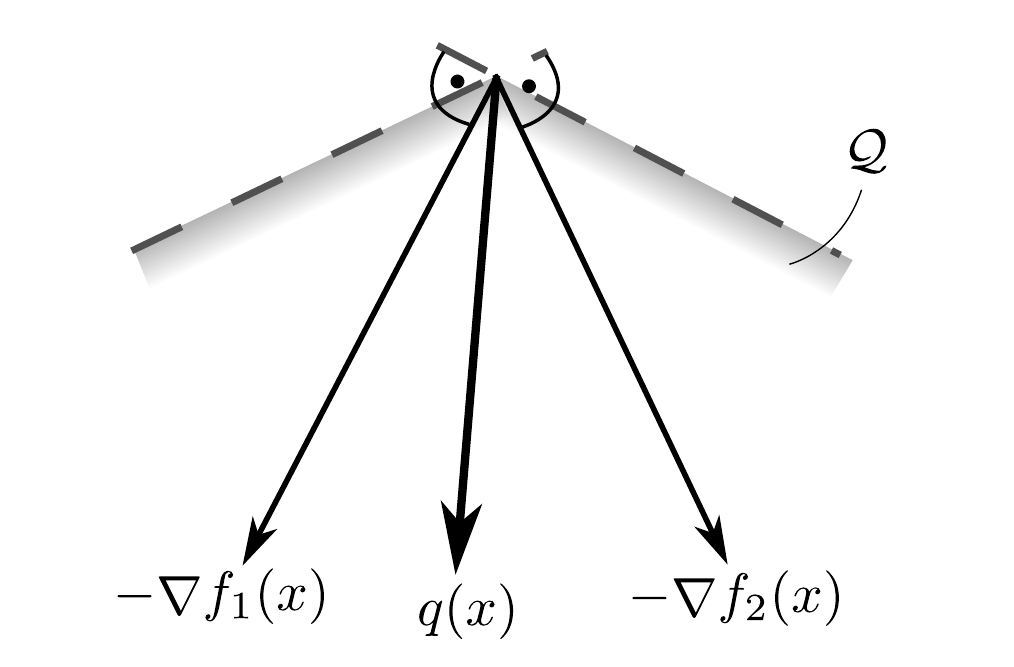}\\(a)}
	\parbox[b]{0.49\textwidth}{\centering \includegraphics[width=0.45\textwidth]{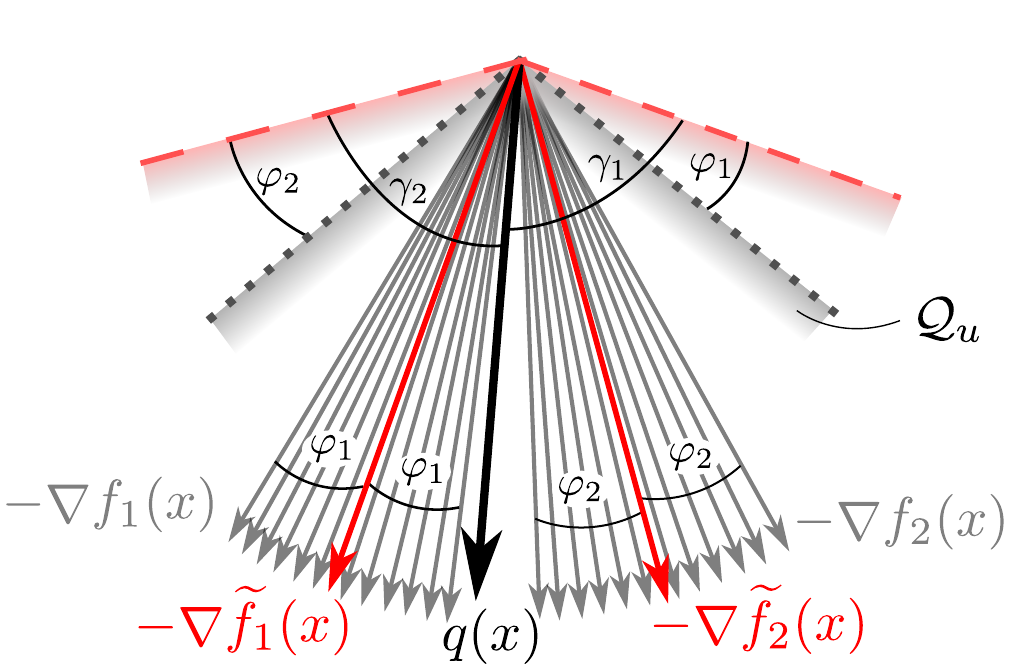}\\(b)}
	\caption{(a) Set of valid descent directions (the exact cone $\SetQ$ bounded by the dashed lines) determined by the intersection of half-spaces defined by the negative gradients. (b) Reduction of the set of valid descent directions in dependence on the errors $\epsilon_i$ (the inexact cone $\SetQr$ bounded by the dotted lines). The gray vectors represent the set of possible values of the exact gradients $\nabla f_i(x)$ and the inexact cone $\SetQr$ is defined by the ``most aligned'' (here the uppermost) realizations of $\nabla f_i(x)$.}
	\label{fig:angle_descent_set}
\end{figure}

We would like to derive an algorithm by which we can determine an inexact descent direction $\qr(x)$ in such a way that it is also valid for the exact problem. To this end, we derive an additional criterion which is equivalent to \eqref{eq:gamma_i}. Concretely, we prove the following lemma:
\begin{lemma} \label{lemma:alpha_min}
	Consider the multiobjective optimization problem \eqref{eq:MOP} with inexact gradient information according to \eqref{eq:gradient_error}. Let
	$\qr(x)$ be an inexact descent direction according to Definition~\ref{def:InexactDescentDirection}. 
	We assume $\|\qr(x)\|_2 \neq 0$, $\|\nabla \ftilde_i(x)\|_2 \neq 0, i = 1, \ldots,  k$.
	Then \eqref{eq:gamma_i} is equivalent to 
	\begin{align}
	\alphahat_{i} &\ge \frac{1}{\|\nabla \ftilde_i(x)\|_2^2} \left(\|\qr(x)\|_2 \epsilon_i - \sum_{\substack{j=1\\j\neq i}}^k \alphahat_j \left(\nabla \ftilde_j(x) \cdot \nabla \ftilde_i(x)\right)\right),
	\quad i=1, \ldots, k.\label{eq:alphahat}
	\end{align}
	In particular, $\qr(x)$ is a descent direction for all objective functions $f_i(x)$
	if \eqref{eq:alphahat} is satisfied.
\end{lemma}

\begin{proof}
	Inserting the expression for $\gamma_i$ in \eqref{eq:def_gamma} into \eqref{eq:gamma_i} yields
	\begin{align}
	\frac{\qr(x) \cdot \left(-\nabla \ftilde_i(x)\right)}{\|\qr(x)\|_2 \cdot \|\nabla \ftilde_i(x)\|_2} \ge \cos\left(\frac{\pi}{2} - \varphi_i\right) = \sin\left(\varphi_i\right) = \frac{\epsilon_i}{\|\nabla \ftilde_i(x)\|_2}. \label{eq:bound_angle_q_gradf}
	\end{align}
	Using the definition of $\qr(x)$ this is equivalent to
	\begin{align}
	\frac{\left(\sum_{j=1}^k \alphahat_j \nabla \ftilde_j(x)\right) \cdot \nabla \ftilde_i(x)}{\|\qr(x)\|_2 \ \|\nabla \ftilde_i(x)\|_2} &= \frac{\alphahat_i \|\nabla \ftilde_i(x)\|_2^2 + \sum_{j=1,j\neq i}^k \alphahat_j \nabla \ftilde_j(x) \cdot \nabla \ftilde_i(x)}{\|\qr(x)\|_2 \ \|\nabla \ftilde_i(x)\|_2} \notag \\
	&= \frac{\|\nabla \ftilde_i(x)\|_2}{\|\qr(x)\|_2} \alphahat_i + \frac{\sum_{j=1,j\neq i}^k \alphahat_j \nabla \ftilde_j(x) \cdot \nabla \ftilde_i(x)}{\|\qr(x)\|_2 \ \|\nabla \ftilde_i(x)\|_2} \notag \\
	&\ge \frac{\epsilon_i}{\|\nabla \ftilde_i(x)\|_2}\, \notag \\
	\Longleftrightarrow \quad \alphahat_{i} \ge \frac{1}{\|\nabla \ftilde_i(x)\|_2^2} &\left(\|\qr(x)\|_2 \epsilon_i - \sum_{\substack{j=1\\j\neq i}}^k \alphahat_j \left(\nabla \ftilde_j(x) \cdot \nabla \ftilde_i(x)\right)\right). \notag
	\end{align}
\end{proof}

\begin{remark} \label{rem:no_error}
	By setting $\epsilon_i = 0$ (i.e.~$\ftilde_i(x) = {f}_i(x)$ for $i=1,\ldots,k$) in \eqref{eq:alphahat} and performing some elemental manipulations, we again obtain the condition \eqref{eq:condition_descent_all} for an exact descent direction:
	\begin{align*}
	\alphahat_{i} &\ge \frac{1}{\|\nabla f_i(x)\|_2^2} \left(- \sum_{\substack{j=1\\j\neq i}}^k \alphahat_j \left(\nabla f_j(x) \cdot \nabla f_i(x)\right)\right) \\
	\Leftrightarrow \alphahat_{i} \left(\nabla f_i(x) \cdot \nabla f_i(x)\right) &\ge \left(- \sum_{\substack{j=1\\j\neq i}}^k \alphahat_j \left(\nabla f_j(x) \cdot \nabla f_i(x)\right)\right) \\
	\Leftrightarrow -\nabla {f}_i(x) \cdot {q}(x) &\ge 0.
	\end{align*}
%	By setting $\epsilon_i = 0$ in \eqref{eq:alphahat} and performing some elemental manipulations, we again obtain the condition for an exact descent direction, i.e.:
%	\[ 
%	-\nabla \ftilde_i(x) \cdot \qr(x) \ge 0 \Leftrightarrow
%	-\nabla {f}_i(x) \cdot {q}(x) \ge 0
%	\]
%	as desired.
\end{remark}
The condition \eqref{eq:alphahat} can be interpreted as a lower bound for the ``impact'' of a particular gradient on the descent direction induced by $\widehat\alpha_i$. The larger the error $\epsilon_i$, the higher the impact of the corresponding gradient needs to be in order to increase the angle between the descent direction and the hyperplane normal to the gradient. The closer a point $x$ is to the Pareto front (i.e.~for small values of $\|q(x)\|_2$), the more confined the region of possible descent directions becomes. Hence, the inaccuracies gain influence until it is no longer possible to guarantee the existence of a descent direction for every objective. This is the case when the sum over the lower bounds from \eqref{eq:alphahat} exceeds one as shown in part (a) of the following result:
\begin{theorem}
	\label{th:inexact_descent}
	Consider the multiobjective optimization problem \eqref{eq:MOP} and suppose that the assumptions
	in Lemma~\ref{lemma:alpha_min} hold. Let
	\[
	\alphahat_{min,i} = \frac{1}{\|\nabla \ftilde_i(x)\|_2^2} \left(\|\qr(x)\|_2 \epsilon_i - \sum_{\substack{j=1\\j\neq i}}^k \alphahat_j \left(\nabla \ftilde_j(x) \cdot \nabla \ftilde_i(x)\right)\right),
	\quad i=1, \ldots, k.
	\]
	Then the following statements are valid:
	\begin{itemize}%\begin{enumerate}[(a)]
		\item[(a)] If $\sum_{i=1}^k \alphahat_{min,i} > 1$ then $\SetQr = \emptyset$
		(see Figure \ref{fig:angle_descent_set} (b)), and
		therefore it cannot be guaranteed that there is a descent direction for all objective functions.
		\item[(b)] All points $x$ with $\sum_{i=1}^k \alphahat_{min,i} = 1$
		are contained in the set
		\begin{align}
			\PSeps = \left\lbrace x \in \R^n ~\Big|~ \Big\|\sum_{i=1}^k\alphahat_i \nabla f_i(x) \Big\|_2 \le 2\| \epsilon \|_{\infty} \right\rbrace. \label{eq:PS_epsilon}
		\end{align}
	\end{itemize}
\end{theorem}

\begin{proof}
For part (a), suppose that we have a descent direction $\qr(x)$ for which $\sum_{i=1}^k \alphahat_{min,i} > 1$. Then by \eqref{eq:alphahat}
	\begin{align*}
		&\sum_{i=1}^k \left( \frac{\|\qr(x)\|_2}{\|\nabla \ftilde_i(x)\|_2^2} \epsilon_i - \frac{\sum_{j=1, j\neq i}^k \alphahat_j \nabla  \ftilde_j(x) \cdot \nabla \ftilde_i(x)}{\|\nabla \ftilde_i(x)\|_2^2} \right) &> 1\color{white}{-\sum_{i=1}^k \alphahat_i=0}\\
		\Leftrightarrow \quad &\sum_{i=1}^k \left(\frac{\|\qr(x)\|_2}{\|\nabla \ftilde_i(x)\|_2^2} \epsilon_i - \frac{\left( -\qr(x) - \alphahat_i \nabla \ftilde_i(x) \right) \cdot \nabla \ftilde_i(x)}{\|\nabla \ftilde_i(x)\|_2^2} \right) &> 1\color{white}{-\sum_{i=1}^k \alphahat_i=0}\\
		\Leftrightarrow \quad &\sum_{i=1}^k \left(\frac{\|\qr(x)\|_2}{\|\nabla \ftilde_i(x)\|_2^2} \epsilon_i - \frac{-\qr(x) \cdot \nabla \ftilde_i(x)}{\|\nabla \ftilde_i(x)\|_2^2} + \alphahat_i \right) &> 1\color{white}{-\sum_{i=1}^k \alphahat_i=0}\\
		\Leftrightarrow \quad &\sum_{i=1}^k \left(\frac{\|\qr(x)\|_2}{\|\nabla \ftilde_i(x)\|_2^2} \epsilon_i - \frac{-\qr(x) \cdot \nabla \ftilde_i(x)}{\|\nabla \ftilde_i(x)\|_2^2} \right) &> 1-\sum_{i=1}^k \alphahat_i=0\\
		\Leftrightarrow \quad &\sum_{i=1}^k \left(\frac{\epsilon_i}{\|\nabla \ftilde_i(x)\|_2} - \frac{-\qr(x) \cdot \nabla \ftilde_i(x)}{\|\qr(x)\|_2 \cdot \|\nabla \ftilde_i(x)\|_2} \right) &> 0\color{white}{-\sum_{i=1}^k \alphahat_i=0}\\
		\Leftrightarrow \quad &\sum_{i=1}^k \sin \varphi_i > \sum_{i=1}^k \sin \gamma_i.
	\end{align*}
Since $\varphi_i,\gamma_i \in [0, \pi/2]$ for $i = 1, \ldots, k$, it follows that $\varphi_i > \gamma_i$ for at least one $i \in \lbrace 1, \ldots, k\rbrace$. This is a contradiction to \eqref{eq:gamma_i} yielding $\SetQr = \emptyset$.

For part (b), we repeat the calculation from part (a) with the distinction that $\sum_{i=1}^k \alphahat_{min,i} = 1$, and obtain $\sum_{i=1}^k \sin \varphi_i = \sum_{i=1}^k \sin \gamma_i$. This implies $\varphi_i = \gamma_i$ for $ i = 1, \ldots, k$, i.e.~the set of descent directions is reduced to a single valid direction. 
This is a situation similar to the one described in Lemma~\ref{lemma:norm_q_tilde}. Having a single valid descent direction results in the fact that there is now a possible realization of the gradients $\nabla f_i(x)$ such that each one is orthogonal to $\qr(x)$. In this situation, $x$ would satisfy \eqref{eq:MOP_optimality} and hence, $\|\qr(x)\|_2 \le \| \epsilon \|_{\infty}$ which leads to
\begin{align*}
	\Big\|\sum_{i=1}^k\alphahat_i \nabla f_i(x) \Big\|_2 &= \Big\|\sum_{i=1}^k\alphahat_i \left( \nabla \ftilde_i(x) + (-\bar{\epsilon}_i) \right) \Big\|_2 = \Big\| \qr(x) + \sum_{i=1}^k\alphahat_i \left( -\bar{\epsilon}_i\right) \Big\|_2 \\
	&\le \Big\| \qr(x) \Big\|_2 + \Big\|\sum_{i=1}^k\alphahat_i \bar{\epsilon}_i \Big\|_2 \le \Big\| \qr(x) \Big\|_2 + \Big\|\sum_{i=1}^k \bar{\epsilon}_i \Big\|_2 \le 2\| \epsilon \|_{\infty}.
\end{align*}
\end{proof}
\begin{corollary}\label{cor:alphamin_gt_1}
	If $\sum_{i=1}^k \alphahat_{min,i} > 1$, the angle of the cone spanned by the exact gradients $\nabla f_i(x)$ is larger than for $\sum_{i=1}^k \alphahat_{min,i} \le 1$ (cf.~Figure~\ref{fig:angle_descent_stop} (b)). 
	Moreover, if $\|\qr(x)\|_2$ is monotonically decreasing for decreasing distances of $x$ to the set of stationary points $\PSsub$, then the result from Theorem~\ref{th:inexact_descent} (b) holds for $\sum_{i=1}^k \alphahat_{min,i} \ge 1$.
\end{corollary}
The results from Theorem \ref{th:inexact_descent} are visualized in Figure~\ref{fig:angle_descent_stop}. In the situation where the inexactness in the gradients $\nabla\ftilde_i(x)$ permits the exact gradients $\nabla f_i(x)$ to satisfy the KKT condition $\eqref{eq:MOP_optimality}$, $\SetQr = \emptyset$ and a descent direction can no longer be computed. 
\begin{figure}[htb]
	\centering
	\parbox[b]{0.7\textwidth}{\centering \includegraphics[width=0.7\textwidth]{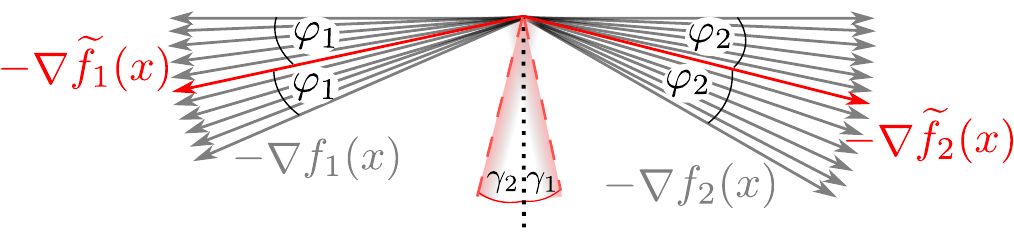}\\(a)}\\~\\~\\
	\parbox[b]{0.7\textwidth}{\centering \includegraphics[width=0.7\textwidth]{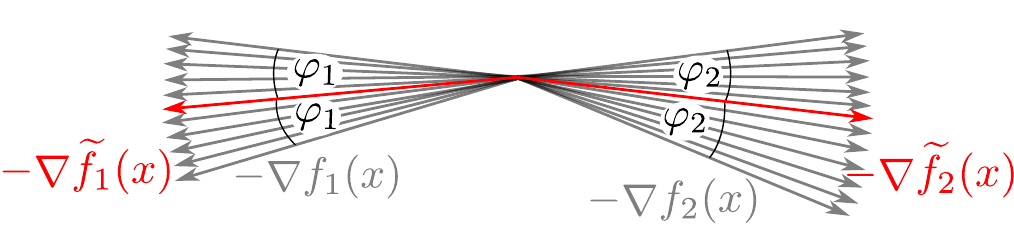}\\(b)}
	\caption{(a) Situation when a valid descent direction can no longer be computed. (b) Situation where $\sum_{i=1}^k \alphahat_{min,i} > 1$.}
	\label{fig:angle_descent_stop}
\end{figure}

In order to numerically compute a valid descent direction, one can compute a direction by solving \eqref{eq:QOP} and consequently verifying that $\qr(x)$ is indeed descending for all objectives by checking condition \eqref{eq:alphahat}. If this is not the case, we can solve \eqref{eq:QOP} again with adjusted lower bounds for $\alphahat$ (cf.~Algorithm~\ref{algo:Descent_direction}).
\begin{algorithm} 
	\caption{(Descent direction for inexact gradients)}
	\label{algo:Descent_direction}
	\begin{algorithmic}[1]
		\Require Inexact gradients $\nabla \ftilde_i(x)$, error bounds $\epsilon_i$ and lower bounds $\alphahat_{min,i} = 0$, $i = 1, \ldots, k$
		\Loop
		\State Compute $\alphahat_i$ by solving \eqref{eq:QOP} with $\alphahat_i \in [\alphahat_{min,i}, 1]$, $i = 1, \ldots,  k$
		\State Evaluate the condition \eqref{eq:alphahat} and update $\alphahat_{min,i}$
		\If{$\sum_{i=1}^{k}\alphahat_{min,i} \ge 1$}\\
			\quad\qquad$\qr(x) = 0$ ($\SetQr = \emptyset$)\\
			\quad\qquad\textbf{STOP}
		\ElsIf{$\alphahat_i \geq \alphahat_{min,i}$, $i = 1, \ldots,  k$}\\
			\quad\qquad$\qr(x) = -\sum_{i=1}^{k}\alphahat_i \nabla \ftilde_i(x)$\\
			\quad\qquad\textbf{STOP}
		\EndIf
		\EndLoop
	\end{algorithmic}
\end{algorithm}
Alternatively, one can compute the entire set of descent directions \cite{Bos12} and chose a direction from this set which satisfies \eqref{eq:gamma_i} or \eqref{eq:alphahat}, respectively. In the examples we investigated, we observed that $\qr(x)$ is equal to $q(x)$ for the majority of points $x$. This is likely due to the fact that by solving \eqref{eq:QOP}, we obtain a steepest descent like direction which very often is relatively far away from the descent cone's boundaries. Close to the set of substationary points, we observe updates of the lower bounds $\alphahat_{min,i}$, but in many cases, the line search strategy directly leads to points where $\SetQr = \emptyset$ without requiring adjusted bounds.

The obtained descent direction can now be utilized to efficiently compute a point which is approximately Pareto optimal, utilizing the advantages of gradient-based methods. In order to compute the entire Pareto set for MOPs with inexact gradient information, we will combine the above result with the algorithm presented in \cite{DSH05} in the next section.

%%%%%%%%%%%%%%%%%%%%%%%%%%%%%%%%%%%%%%%%%%%%%%%%%%%%%%%%%%%%%%%%%%%%%%%%%%%%%%%%%%%%%%%%%%%%%%
%% Subdivision Algorithm
%%%%%%%%%%%%%%%%%%%%%%%%%%%%%%%%%%%%%%%%%%%%%%%%%%%%%%%%%%%%%%%%%%%%%%%%%%%%%%%%%%%%%%%%%%%%%%
\section{Subdivision Algorithm}
\label{sec:Subdivision}
Using the result of Theorem~\ref{th:inexact_descent} we can construct a global subdivision algorithm which computes a nested sequence of increasingly refined box coverings of the entire set of \emph{almost substationary points} $\PSeps$ (Equation \eqref{eq:PS_epsilon}). Besides globality, a benefit of this technique is that it can easily be applied to higher dimensions whereas especially geometric approaches struggle with a larger number of objectives. The computational cost, however, increases exponentially with the dimension of the Pareto set such that in practice, we are restricted to a moderate number of objectives. 
In the following, we recall the subdivison algorithm for the solution of multiobjective optimization problems with exact information before we proceed with our extension to inexact information. For details we refer to \cite{DSH05}.

%%%%%%%%%%%%%%%%%%%%%%%%%%%%%%%%%%%%%%%%%%%%%%%%%%%%%%%%%%%%%%%%%%%%%%%%%%%%%%%%%%%%%%%%%%%%%%
%% Subsection: Subdivision Algorithm with exact gradients
%%
\subsection{Subdivision Algorithm with Exact Gradients}
\label{subsec:Subdivision_exact}
In order to apply the subdivision algorithm to a multiobjective optimization problem, we first formulate a descent step of the optimization procedure using a line search approach, i.e.
\begin{align}
x_{j+1} = g(x_j) = x_j + h_j q_j(x_j), \label{eq:dyn_sys_MO}
\end{align}
where $q_j(x_j)$ is the descent direction according to \eqref{eq:q} and $h_j$ is an appropriately chosen step length (e.g.~according to the Armijo rule \cite{NW06} for all objectives). The subdivision algorithm was initially developed in the context of dynamical systems \cite{DH97} in order to compute the global attractor of a dynamical system $g$ relative to a set $Q$, i.e.~the set $A_Q$ such that $g(A_Q) = A_Q$. 
Using a multilevel subdivision scheme, the following algorithm yields an outer approximation of $A_{Q}$ in the form of a sequence of sets $\mathcal{B}_0,\mathcal{B}_1,\ldots$, where each $\mathcal{B}_s$ is a subset of $\mathcal{B}_{s-1}$ and consists of finitely many subsets $B$ (from now on referred to as \emph{boxes}) of $Q$ covering the relative global attractor $A_{Q}$. For each set $\mathcal{B}_s$, we define a box diameter 
\begin{align*}
\text{diam}(\mathcal{B}_s) = \max_{B \in \mathcal{B}_s} \text{diam}(B)
\end{align*}
which tends to zero for $s\rightarrow \infty$ within Algorithm~\ref{algo:Subdivision}.
\begin{algorithm} 
	\caption{(Subdivision algorithm)}
	\label{algo:Subdivision}
	Let $\mathcal{B}_0$ be an initial collection of finitely many subsets of the compact set $Q$ such that $\bigcup_{B\in\mathcal{B}_0}B = Q$. Then, $\mathcal{B}_s$ is inductively obtained from $\mathcal{B}_{s-1}$ in two steps:
	\begin{itemize}
		\item[(i)] \textbf{Subdivision}. Construct from $\mathcal{B}_{s-1}$ a new collection of subsets $\widehat{\mathcal{B}}_s$ such that
		\begin{align*}
		\bigcup_{B\in\widehat{\mathcal{B}}_s} B &= \bigcup_{B\in \mathcal{B}_{s-1}} B, \\
		\text{diam}(\widehat{\mathcal{B}}_s) &= \theta_s \text{diam}(\mathcal{B}_{s-1}), \quad 0 < \theta_{min}\le \theta_s \le \theta_{max} < 1.
		\end{align*}
		\item[(ii)] \textbf{Selection}. Define the new collection $\mathcal{B}_s$ by
		\begin{align*}
		\mathcal{B}_s = \left\lbrace B \in \widehat{\mathcal{B}}_s ~ \Big| ~ \exists \widehat{B}\in \widehat{\mathcal{B}}_s \text{ such that } g^{-1}(B) \cap \widehat{B} \neq \emptyset \right\rbrace.
		\end{align*}
	\end{itemize}
\end{algorithm}
Interpreting \eqref{eq:dyn_sys_MO} as a dynamical system, the attractor is the set of points for which $q(x_j) = 0$, i.e.~the set of points satisfying the KKT conditions. As a first step, we can prove that each accumulation point of the system is a substationary point for \eqref{eq:MOP}:
\begin{theorem} [\cite{DSH05}] \label{th:accumulation_q}
	Suppose that $x^*$ is an accumulation point of the sequence $(x_j)_{j=0,1,\ldots}$ created by \eqref{eq:dyn_sys_MO}. Then, $x^*$ is a substationary point of \eqref{eq:MOP}.
\end{theorem}
Using Algorithm \ref{algo:Subdivision}, we can now compute an outer approximation of the attractor of the dynamical system \eqref{eq:dyn_sys_MO} which contains all points satisfying \eqref{eq:MOP_optimality}. The attractor of a dynamical system is always connected, which is not necessarily the case for $\PSsub$. In this situation, the attractor is a superset of $\PSsub$. However, if $\PSsub$ is bounded and connected, it coincides with the attractor of \eqref{eq:dyn_sys_MO}, which is stated in the following theorem:
\begin{theorem}[\cite{DSH05}]\label{th:Subdivision}
	Suppose that the set $\PSsub$ of points $x\in\R^n$ satisfying \eqref{eq:MOP_optimality} is bounded and connected. Let $Q$ be a compact neighborhood of $\PSsub$. Then, an application of Algorithm~\ref{algo:Subdivision} to $Q$ with respect to the iteration scheme \eqref{eq:dyn_sys_MO} leads to a sequence of coverings $\mathcal{B}_s$ which converges to the entire set $\PSsub$, that is,
	\begin{align*}
		d_h(\PSsub, \mathcal{B}_s) \rightarrow 0, \quad \text{for } s = 0,1,2, \ldots,
	\end{align*}
	where $d_h$ denotes the Hausdorff distance.
\end{theorem}

The concept of the subdivision algorithm is illustrated in Figure~\ref{fig:GAIO_grad}, where in addition a numerical realization has been introduced (cf.~Section~\ref{subsubsec:Numerical_realization} for details).

%%%%%%%%%%%%%%%%%%%%%%%%%%%%%%%%%%%%%%%%%%%%%%%%%%%%%%%%%%%%%%%%%%%%%%%%%%%%%%%%%%%%%%%%%%%%%%
%% Subsection: Subdivision Algorithm with inexact gradients
%%
\subsection{Subdivision Algorithm with Inexact Function and Gradient Values}
\label{subsec:Subdivision_inexact}
In this section, we combine the results from \cite{DSH05} with the results from Section~\ref{sec:InexactGradients} in order to devise an algorithm for the approximation of the set of substationary points of \eqref{eq:MOP} with inexact function and gradient information. So far, we have only considered inexactness in the descent direction where we only need to consider errors in the gradients. For the computation of a descent step, we further require a step length strategy \cite{NW06} where we additionally need to consider errors in the function values. For this purpose, we extend the concept of \emph{non-dominance} (\ref{def:Pareto_optimality}) to inexact function values:
\begin{definition}
	\label{def:Pareto_optimality_inexact}
	Consider the multiobjective optimization problem \eqref{eq:MOP}, where the objective functions $f_i(x)$, $i=1, \ldots, k$, are only known approximately according to \eqref{eq:function_error}. 
	Then 
	\begin{itemize}
		\item[(a)] a point $x^* \in \R^n$ \emph{confidently dominates} a point $x \in \R^n$, if $\ftilde_i(x^*)+\xi_i \leq \ftilde_i(x)-\xi_i$ for $i = 1, \ldots, k$ and $\ftilde_i(x^*)+\xi_i < \ftilde_i(x)-\xi_i$ for at least one $i \in 1, \ldots, k$.
		\item[(b)] a set $\mathcal{B}^* \subset \R^n$ \emph{confidently dominates} a set $\mathcal{B} \subset \R^n$ if for every point $x \in \mathcal{B}$ there exists at least one point $x^* \in \mathcal{B}^*$ dominating $x$.
		\item[(c)] The \emph{set of almost non-dominated points} which is a superset of the Pareto set $\PS$ is defined as:
			\begin{align}
				\PSxi = \left\lbrace x^* \in \R^n \Big | \nexists x \in \R^n \text{ with } \ftilde_i(x)+\xi_i \leq \ftilde_i(x^*)-\xi_i,~i = 1, \ldots, k \right\rbrace. \label{eq:PS_xi}
			\end{align}
	\end{itemize}
\end{definition}
\noindent
Note that the same definition was also introduced in \cite{SVCC09} in order to increase the number of almost Pareto optimal points and thus, the number of possible options for a decision maker. 

As a consequence of the inexactness in the function values and the gradients, the approximated set is a superset of the Pareto set. Depending on the errors $\xi_i$ and $\epsilon_i$, $i = 1, \ldots, k$, in the function values and in the gradients, respectively, each point is either contained in the set $\PSxi$ \eqref{eq:PS_xi} or in the set $\PSeps$ \eqref{eq:PS_epsilon}. Based on these considerations, we introduce an inexact dynamical system similar to \eqref{eq:dyn_sys_MO}:
\begin{align} 
    x_{j+1} &= x_j + h_j p_j, \label{eq:dyn_sys_inexact}
\end{align}
where the direction $p_j$ is computed using Algorithm~\ref{algo:Descent_direction} ($p_j = \qr(x_j))$ and the step length $h_j$ is determined by a modified Armijo rule \cite{NW06} such that $\ftilde_i(x_j + h_j p_j) + \xi_i \leq \ftilde_i(x_{j-1}) + c_1 h_j p_j^\top \nabla \ftilde_i(x_{j-1})$.
If the errors $\xi$ and $\epsilon$ are zero, the computed set is reduced to the set of substationary points. In this situation $\PSxi = \PS$ and $\sum_{i=1}^k \alphahat_{min,i} = 0$, hence $p_j = q(x_j)$. Convergence of the dynamical system \eqref{eq:dyn_sys_inexact} to an \emph{approximately substationary} point is investigated in the following theorem:
\begin{theorem} \label{th:accumulation_inexact}
	Consider the multiobjective optimization problem \eqref{eq:MOP} with inexact objective functions and inexact gradients according to \eqref{eq:function_error} and \eqref{eq:gradient_error}, respectively. Suppose that $x^*$ is an accumulation point of the sequence $(x_j)_{j=0,1,\ldots}$ created by \eqref{eq:dyn_sys_inexact}. Then,
	\begin{itemize}%\begin{enumerate}[(a)]
		\item[(a)] $x^* \in \PSepsxi = \PSeps \cup \PSxi$ where $\PSeps$ and $\PSxi$ are defined according to \eqref{eq:PS_epsilon} and \eqref{eq:PS_xi}, respectively.
		\item[(b)] If $\xi_i = 0$, $\epsilon_i = 0$, $i=1, \ldots, k$, $x^*$ is a substationary point of \eqref{eq:MOP}.
	\end{itemize}
\end{theorem}
\begin{proof}
(a) For a point $x_j$ created by the sequence \eqref{eq:dyn_sys_inexact} one of the following statements is true:
\begin{align*}
	\mbox{i)}\quad x_j \in \PSeps \quad \land \quad x_j \notin \PSxi \qquad \qquad \mbox{ii)}\quad x_j \notin \PSeps \quad \land \quad x_j \in \PSxi\\
	\mbox{iii)} \quad x_j \in \PSeps \quad \land \quad x_j \in \PSxi \qquad \qquad \mbox{iv)} \quad x_j \notin \PSeps \quad \land \quad x_j \notin \PSxi\\
\end{align*}
In case i) $x_j \in \PSeps$ which means that the gradients $\nabla \ftilde_i(x_j)$, $i=1, \ldots,k$, approximately satisfy the KKT conditions. We obtain $\sum_{i=1}^k \alphahat_{min,i} = 1$, i.e.~the set of valid descent directions is empty ($\SetQr=\emptyset$, cf.~Theorem~\ref{th:inexact_descent}). Consequently, $p_j = 0$ and the point $x_j$ is an accumulation point of the sequence \eqref{eq:dyn_sys_inexact}. In case ii) the inaccuracies in the function values $\ftilde_i(x_j)$, $i=1, \ldots,k$ prohibit a guaranteed decrease for all objectives. According to the modified Armijo rule $h_j = 0$ such that $x_j$ is an accumulation point. In case iii) both $p_j = 0$ and $h_j = 0$. In case iv) we have $p_j = q(x_j)$. If for any $j \in \lbrace 0, 1, \ldots \rbrace$, $x_j \in \PSeps$ or $x_j \in \PSxi$, we are in one of the cases i) to iii) and $x_j$ is an accumulation point. Otherwise, we obtain a descent direction such that the sequence \eqref{eq:dyn_sys_inexact} converges to a substationary point $x^* \in \PSsub \subseteq \PSeps$ of \eqref{eq:MOP} which is proved in \cite{DSH05}.

For part (b), we obtain $\alphahat_{min,i} = 0$ by setting the errors $\epsilon_i$, $i=1, \ldots, k$, to zero and hence, the descent direction is $p_j = q(x_j)$ (cf.~Algorithm.~\ref{algo:Descent_direction}). When $\xi_i = 0$, $i=1, \ldots, k$, the modified Armijo rule becomes the standard Armijo rule for multiple objectives. Consequently, the problem is reduced to the case with exact function and gradient values (case iv) in part (a)).
\end{proof}

Following along the lines of \cite{DSH05}, we can use this result in order to prove convergence of the subdivision algorithm with inexact values:
\begin{theorem} \label{th:Subdivision_inexact}
	Suppose that the set $\PSepsxi = \PSeps \cup \PSxi$ where $\PSeps$ and $\PSxi$ are defined according to \eqref{eq:PS_epsilon} and \eqref{eq:PS_xi}, respectively, is bounded and connected. Let $Q$ be a compact neighborhood of $\PSepsxi$. Then, an application of Algorithm~\ref{algo:Subdivision} to $Q$ with respect to the iteration scheme \eqref{eq:dyn_sys_inexact} leads to a sequence of coverings $\mathcal{B}_s$ which is a subset of $\PSepsxi$ and a superset of the set $\PSsub$ of substationary points $x\in\R^n$ of \eqref{eq:MOP}, that is,
	\begin{align*}
		\PSsub \subset \mathcal{B}_s \subset \PSepsxi.
	\end{align*}
	Consequently, if the errors tend towards zero, we observe
	\begin{align*}
		\lim_{\epsilon_i, \xi_i \rightarrow 0,~i = 1, \ldots, k} d_h(\mathcal{B}_s, \PSepsxi) = d_h(\PSsub, \mathcal{B}_s) =  0.
	\end{align*}
\end{theorem}

%%%%%%%%%%%%%%%%%%%%%%%%%%%%%%%%%%%%%%%%%%%%%%%%%%%%%%%%%%%%%%%%%%%%%%%%%%%%%%%%%%%%%%%%%%%%%%
%% Subsection: Numerical realization of the selection step
%%
\subsubsection{Numerical Realization of the Selection Step}
\label{subsubsec:Numerical_realization}
In this section, we briefly describe the numerical realization of Algorithm \ref{algo:Subdivision}. For details, we refer to \cite{DSH05}. The elements $B \in \mathcal{B}_s$ are $n$-dimensional boxes. In the selection step, each box is represented by a prescribed number of sample points at which the dynamical system \eqref{eq:dyn_sys_inexact} is evaluated according to Algorithm~\ref{algo:descent_step_inexact} (see Figure~\ref{fig:GAIO_grad}~(a) for an illustration). Then, we evaluate which boxes the sample points are mapped into and eliminate all ``empty'' boxes, i.e.~boxes which do not possess a preimage within $\mathcal{B}_s$ (Figure~\ref{fig:GAIO_grad}~(b)). The remaining boxes are subdivided and we proceed with the next elimination step until a certain stopping criterion is met (e.g.~a prescribed number of subdivision steps).
\begin{algorithm}[H]
\caption{(Descent step under inexactness)}
\label{algo:descent_step_inexact}
\begin{algorithmic}[1]
\Require Initial point $x_0$, error bounds $\xi_i$ and $\epsilon_i$, $i = 1, \ldots, k$, constant $0 < c_1 < 1$
\State Compute $\ftilde_i(x_0)$ and $\nabla \ftilde_i(x_0)$, $i = 1, \ldots, k$
\State Compute a direction $p$ according to Algorithm~\ref{algo:Descent_direction}
\State Compute a step length $h$ that satisfies the modified Armijo rule \newline $\ftilde_i(x_0 + h p) + \xi_i \leq \ftilde_i(x_0) - \xi_i + c_1 h p^\top \nabla \ftilde_i(x_0)$, e.g.~via backtracking \cite{NW06}
\State Compute $x_{step} = x_0 + h p$
\end{algorithmic}
\end{algorithm}

\begin{figure}[t]
	\centering
	\parbox[b]{0.245\textwidth}{\centering \includegraphics[width=0.24\textwidth]{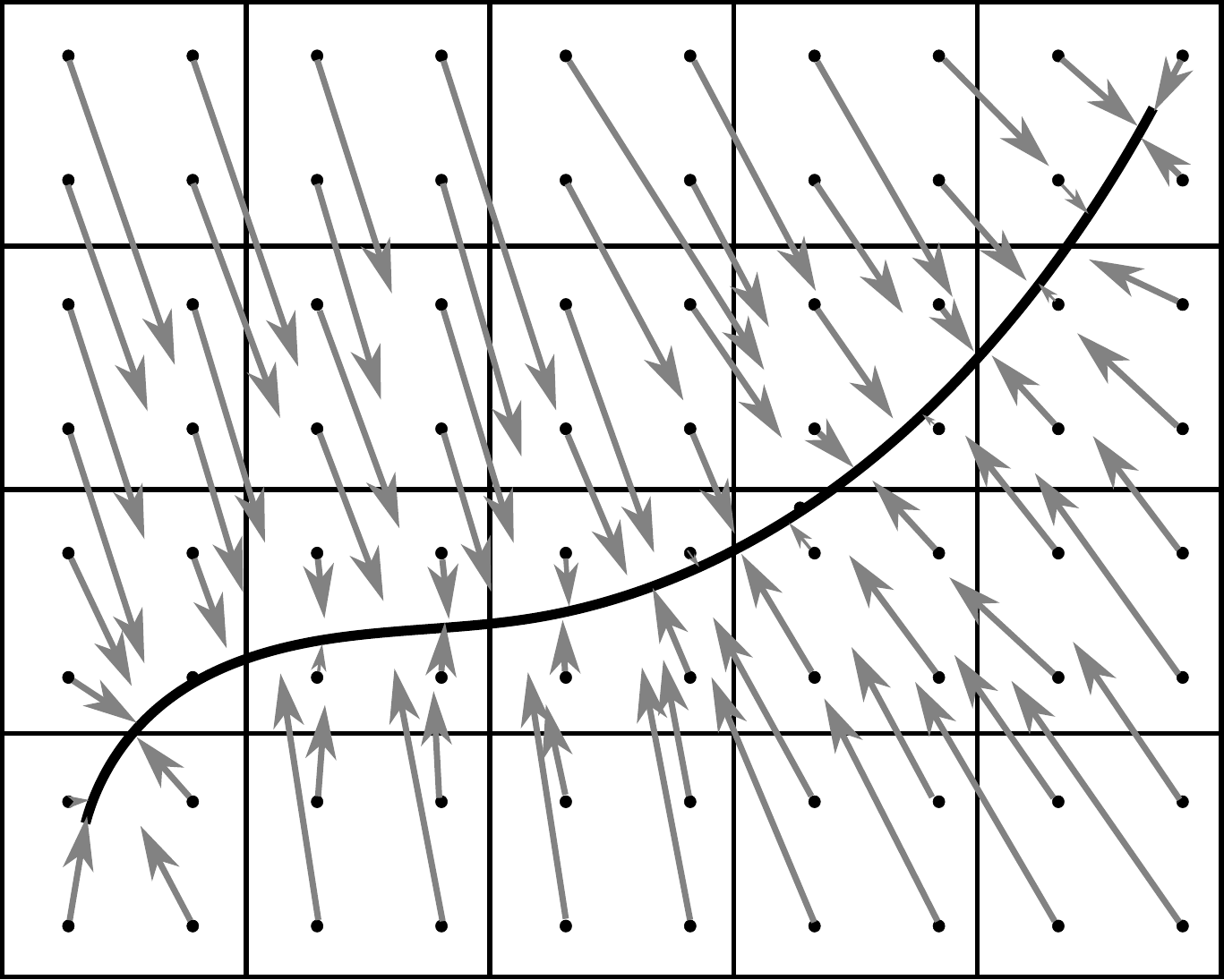}\\(a)}
	\parbox[b]{0.245\textwidth}{\centering \includegraphics[width=0.24\textwidth]{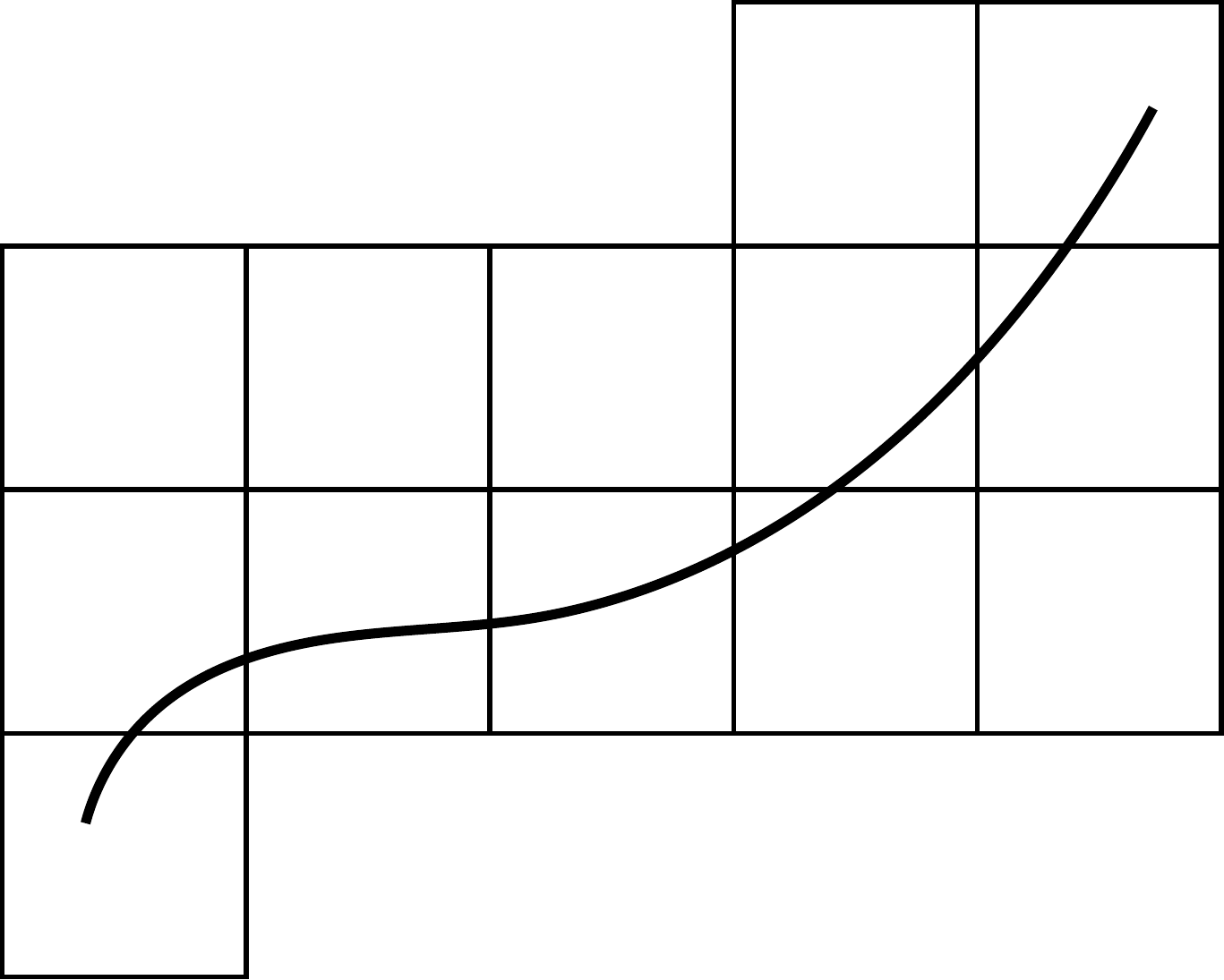}\\(b)}
	\caption{Global subdivision algorithm -- selection step. (a) Evaluation of the dynamical system \eqref{eq:dyn_sys_MO}. (b) All boxes that do not possess a preimage
		within the collection are discarded.}
	\label{fig:GAIO_grad}
\end{figure}

%%%%%%%%%%%%%%%%%%%%%%%%%%%%%%%%%%%%%%%%%%%%%%%%%%%%%%%%%%%%%%%%%%%%%%%%%%%%%%%%%%%%%%%%%%%%%%
%% Subsection: Gradient-free realization
%%
\subsection{Gradient Free Realization}
\label{subsec:Gradient_free_realization}
In many applications, gradients are unknown or difficult to compute. In this case, we can use a gradient-free alternative of Algorithm \ref{algo:Subdivision} which is called the \emph{Sampling algorithm} in \cite{DSH05}. Algorithm \ref{algo:Sampling} also consists of a subdivision and a selection step with the difference that the selection step is a non-dominance test. Hence, we compare all sample points and eliminate all boxes that contain only dominated points. This way, it is also possible to easily include constraints. In the presence of inequality constraints, for example, we eliminate all boxes for which all sample points violate the constraints and then consequently perform the non-dominance test on the remaining boxes. Equality constrains are simply modeled by introducing two inequality constraints.
Finally, a combination of both the gradient-based and the gradient-free algorithm can be applied in order to speed up convergence or to reduce the gradient-based algorithm to the computation of the Pareto set $\PS$ instead of the set of substationary points $\PSsub$.

Considering errors in the function values, we can use the sampling algorithm to compute the superset $\PSxi$ of the global Pareto set $\PS$. In the limit of vanishing errors, this is again reduce to the exact Pareto set:
\[
	\lim_{\xi_i\rightarrow 0,~i=1, \ldots, k} d_h(\PSxi,\PS) = 0.
\]

\begin{algorithm} 
	\caption{(Sampling algorithm)}
	\label{algo:Sampling}
	Let $\mathcal{B}_0$ be an initial collection of finitely many subsets of the compact set $Q$ such that $\bigcup_{B\in\mathcal{B}_0}B = Q$. Then, $\mathcal{B}_s$ is inductively obtained from $\mathcal{B}_{s-1}$ in two steps:
	\begin{itemize}
		\item[(i)] Subdivision. Construct from $\mathcal{B}_{s-1}$ a new collection of subsets $\widehat{\mathcal{B}}_s$ such that
		\begin{align*}
			\bigcup_{B\in\widehat{\mathcal{B}}_s} B &= \bigcup_{B\in \mathcal{B}_{s-1}} B, \\
			\text{diam}(\widehat{\mathcal{B}}_s) &= \theta_s \text{diam}(\mathcal{B}_{s-1}), \quad 0 < \theta_{min}\le \theta_s \le \theta_{max} < 1.
		\end{align*}
		\item[(ii)] Selection. Define the new collection $\mathcal{B}_s$ by
		\begin{align*}
			\mathcal{B}_s = \left\lbrace B \in \widehat{\mathcal{B}}_s ~ \Big| ~ \nexists \widehat{B}\in \widehat{\mathcal{B}}_s \text{ such that } \widehat{B} \text{ confidently dominates } B \right\rbrace.
		\end{align*}
	\end{itemize}
\end{algorithm}

%%%%%%%%%%%%%%%%%%%%%%%%%%%%%%%%%%%%%%%%%%%%%%%%%%%%%%%%%%%%%%%%%%%%%%%%%%%%%%%%%%%%%%%%%%%%%%
%% Results
%%%%%%%%%%%%%%%%%%%%%%%%%%%%%%%%%%%%%%%%%%%%%%%%%%%%%%%%%%%%%%%%%%%%%%%%%%%%%%%%%%%%%%%%%%%%%%
\section{Results}
\label{sec:Results}
In this section, we illustrate the results from Sections~\ref{sec:InexactGradients} and \ref{sec:Subdivision} using three examples. To this end, we add random perturbations to the respective model such that \eqref{eq:function_error} and \eqref{eq:gradient_error} hold. We start with a two dimensional example function $F:\R^2 \rightarrow \R^2$ for two paraboloids:
\begin{align}
	\min_{x\in\R^2} F(x) = \min_{x\in\R^2} \left( \begin{array}{c}
	(x_1 - 1)^2 + (x_2 - 1)^4 \\
	(x_1 + 1)^2 + (x_2 + 1)^2
	\end{array} \right). \label{eq:example_zweiparabeln}
\end{align}
In Figure~\ref{fig:zweiparabeln}, the box covering of the Pareto set and the corresponding Pareto front obtained with Algorithm~\ref{algo:Subdivision} are shown without errors and with $\xi = (0, 0)^{\top}$, $\epsilon = (0.1, 0.1)^{\top}$ (Figure~\ref{fig:zweiparabeln} (a)) and $\xi = (0, 0)^{\top}$, $\epsilon = (0.0, 0.2)^{\top}$ (Figure~\ref{fig:zweiparabeln} (b)), respectively. The background in (a) and (b) is colored according to the norm of the optimality condition \eqref{eq:MOP_optimality}, obtained by solving \eqref{eq:QOP}, and the white line indicates the upper bound of the error \eqref{eq:PS_epsilon}. We see that in (a), the box covering is close to the error bound whereas it is less sharp in (b). Consequently, the error estimate is more accurate when the errors are of comparable size in all gradients. 
The Pareto fronts corresponding to (a) and (b) are shown in (c) and (d). We see that the difference between the Pareto front of the exact solution (red) and the Pareto front of the inexact solution (green) is relatively small but that additional points are computed at the boundary of the front, i.e.~close to the individual minima $F_1(x) = 0$ and $F_2(x) = 0$.
\begin{figure}
	\centering
	\hspace{0.04\textwidth}
	\parbox[b]{0.47\textwidth}{\centering \includegraphics[width=0.47\textwidth]{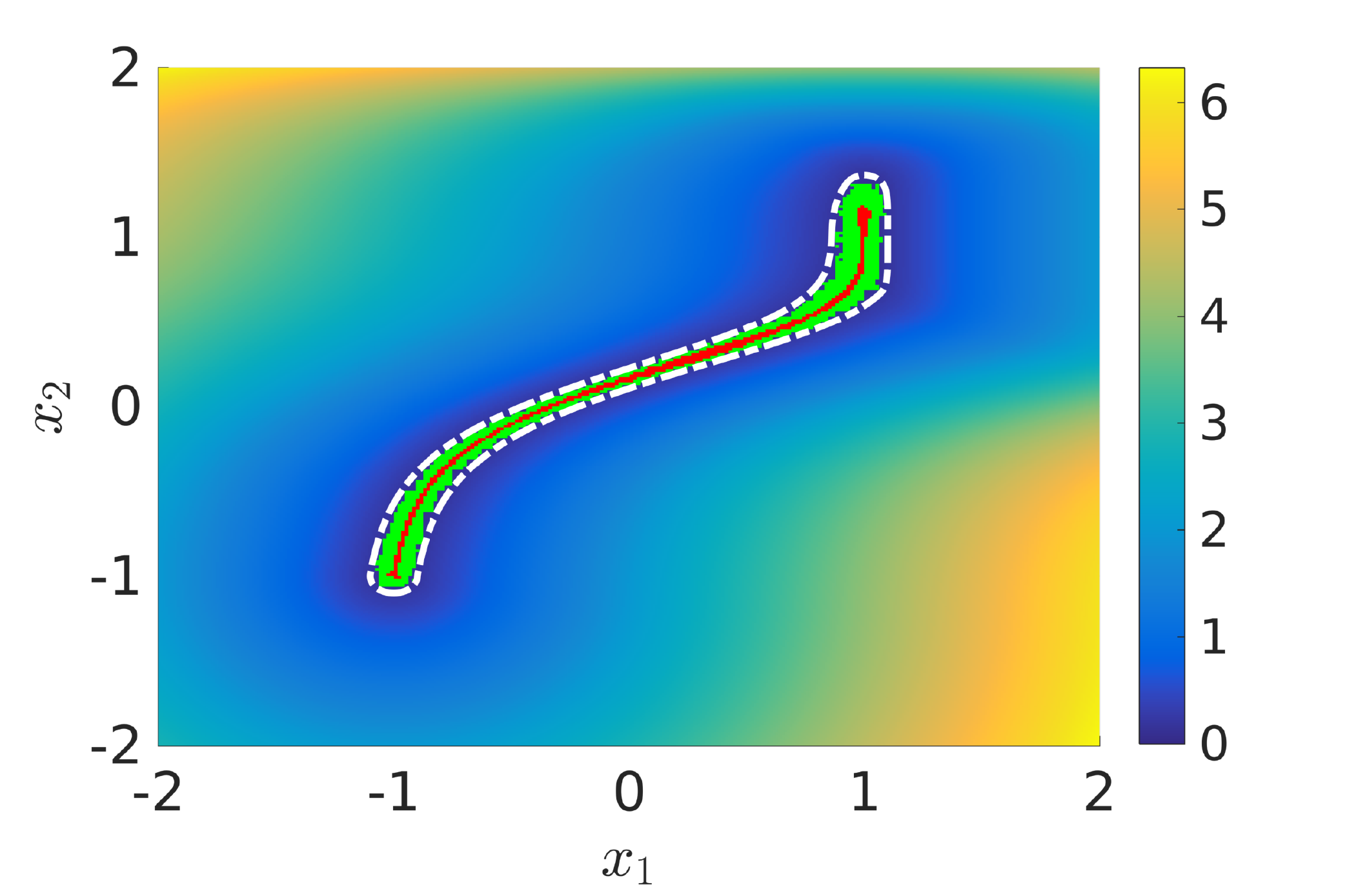}\\(a)}
	\parbox[b]{0.47\textwidth}{\centering \includegraphics[width=0.47\textwidth]{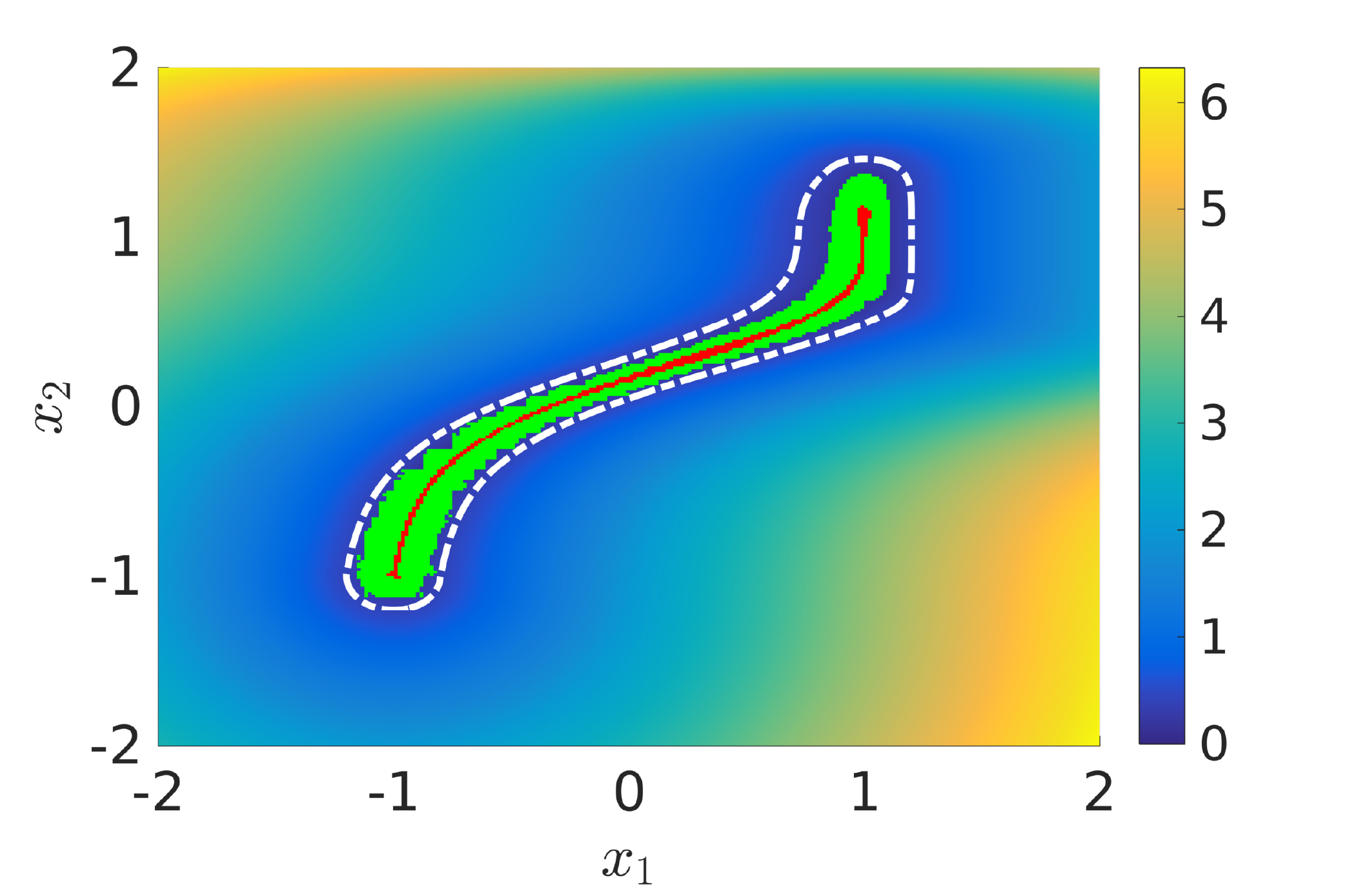}\\(b)} \\
	\parbox[b]{0.4\textwidth}{\centering \includegraphics[width=0.4\textwidth]{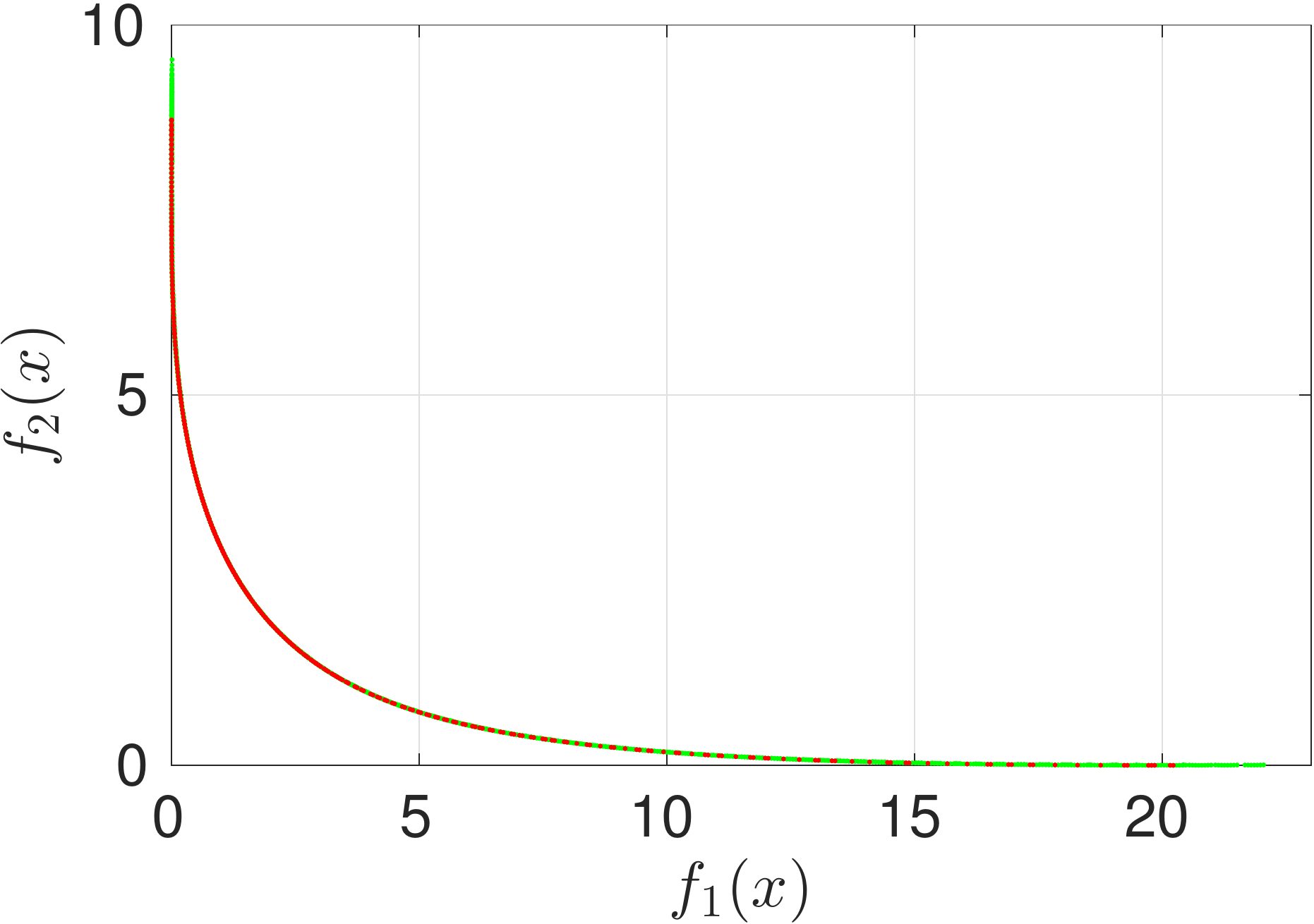}\\(c)} \hspace{0.06\textwidth}
	\parbox[b]{0.4\textwidth}{\centering \includegraphics[width=0.4\textwidth]{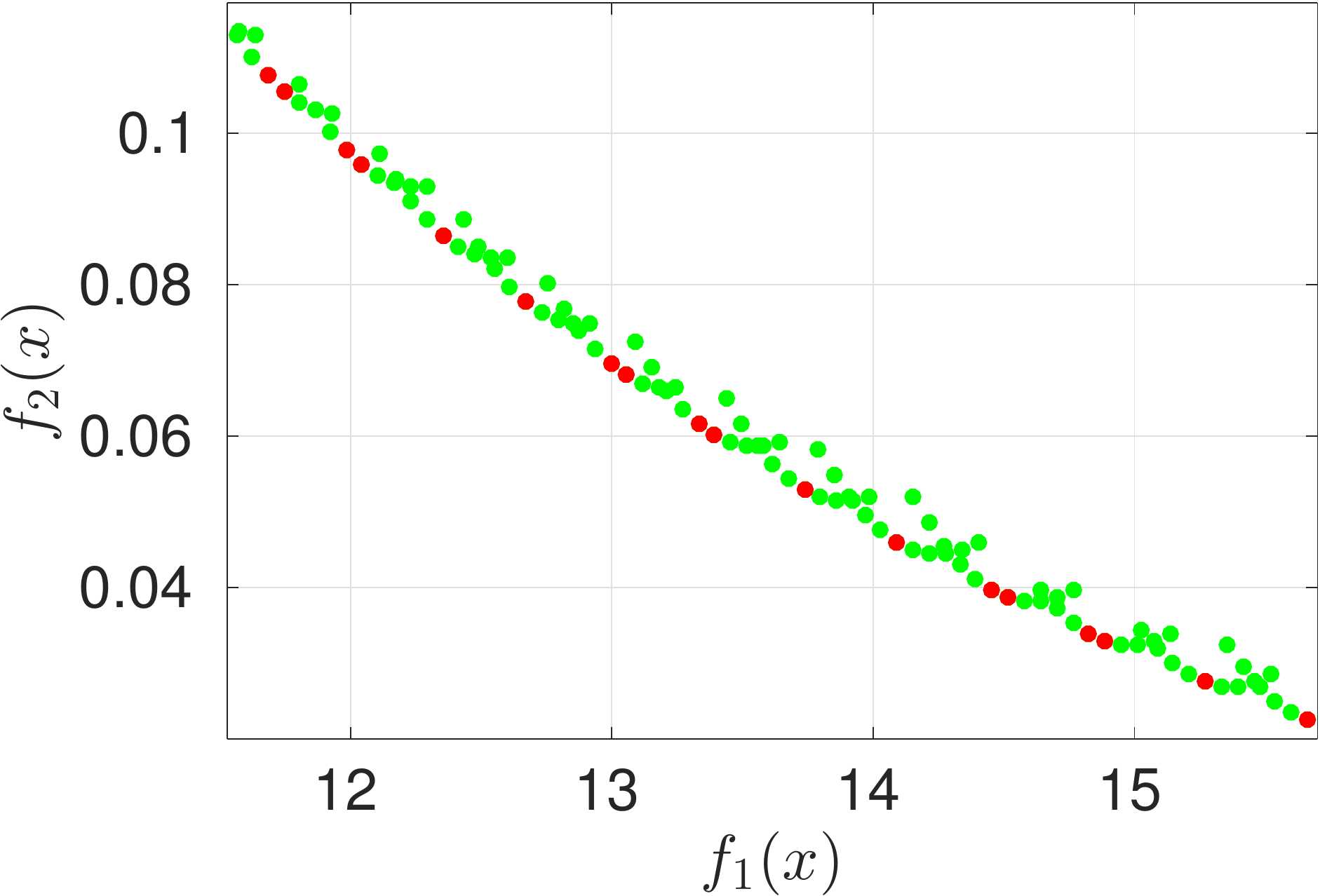}\\(d)}
	\caption{(a) Box covering of the Pareto set of problem~\eqref{eq:example_zweiparabeln} after $16$ subdivision steps (diam$(\mathcal{B}_0) = 4$, diam$(\mathcal{B}_{20}) = 1/2^{6}$). The solution without errors ($\PSsub$) is shown in red, the solution with $\epsilon = (0.1, 0.1)^{\top}$, $\xi = (0, 0)^{\top}$ ($\PSeps$) is shown in green. The background color represents the norm of the optimality condition \eqref{eq:MOP_optimality} and the white line is the iso-curve $\|q(x)\|_2 = 2 \|\epsilon\|_{\infty} = 0.2$, i.e.~the upper bound of the error. (b) Analog to (a) but with $\epsilon = ( 0, 0.2 )^{\top}$ and the iso-curve $\|q(x)\|_2 = 0.4$. (c)--(d) The Pareto fronts corresponding to (a). The points are the images of the box centers (color coding as in (a) and (b)).}
\label{fig:zweiparabeln}
\end{figure}
\begin{figure}
	\centering
	\parbox[b]{0.45\textwidth}{\centering \includegraphics[width=0.45\textwidth]{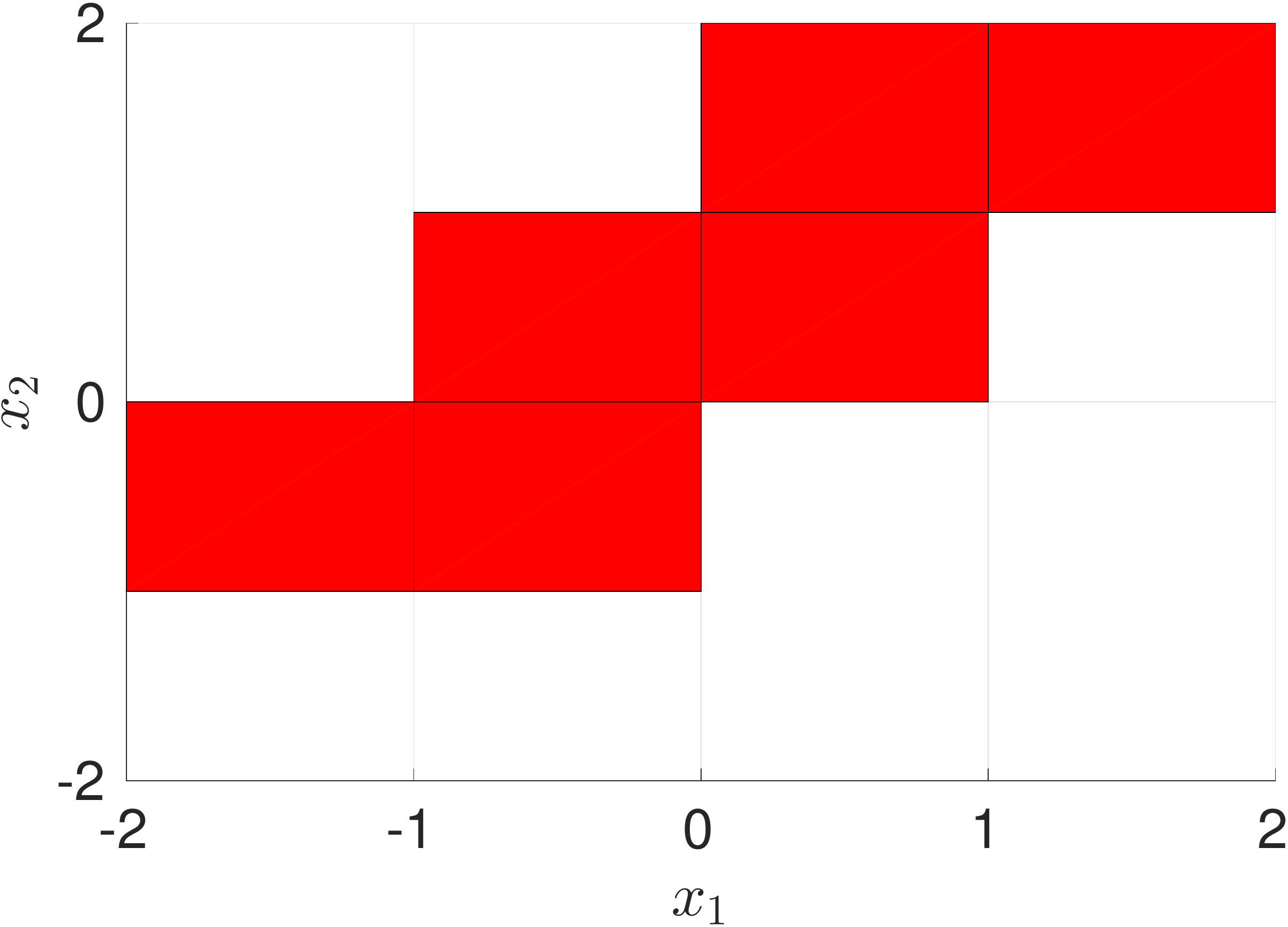}\\(a)} \hspace{0.06\textwidth}
	\parbox[b]{0.45\textwidth}{\centering \includegraphics[width=0.45\textwidth]{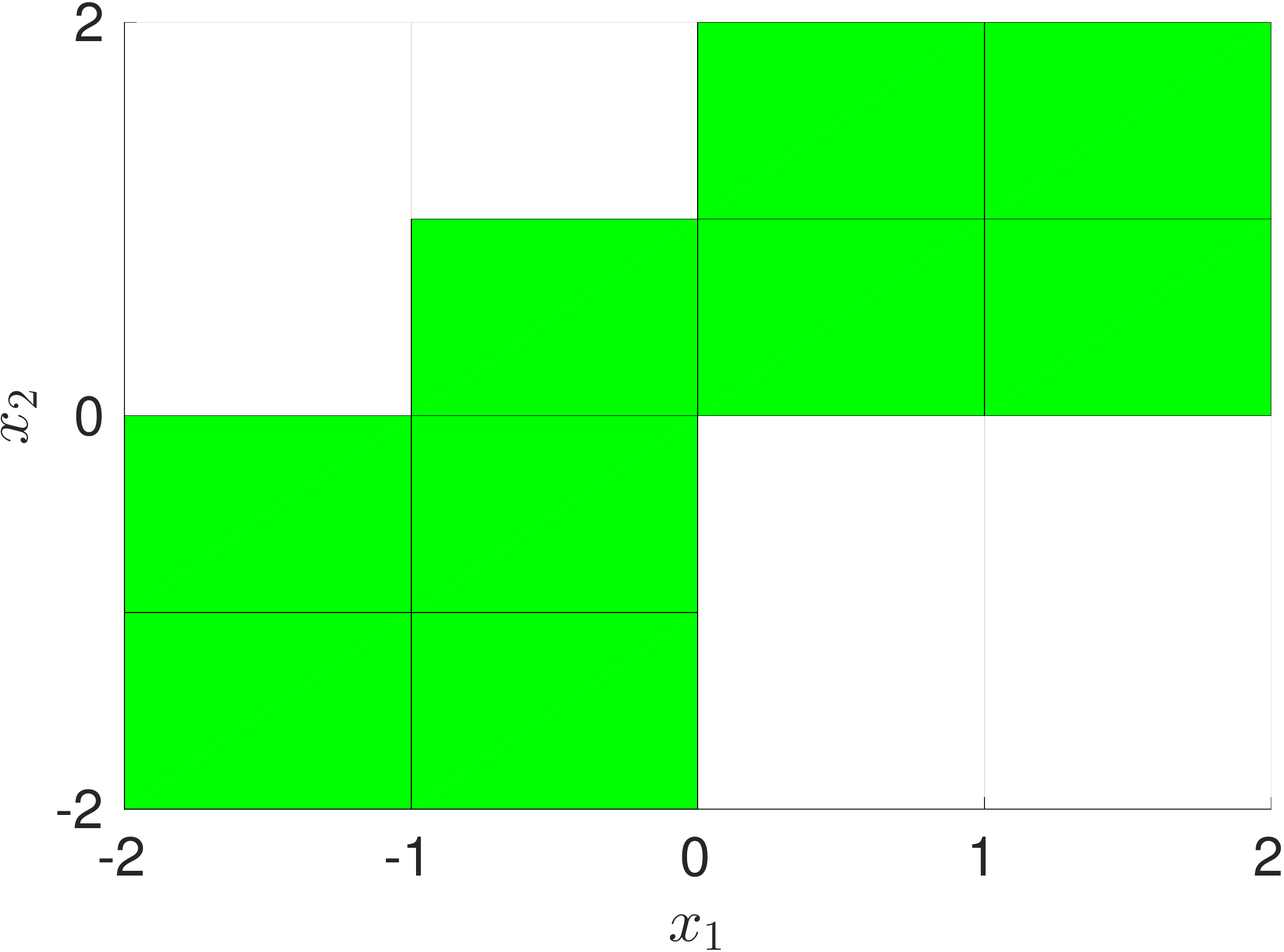}\\(b)} \\
%	\parbox[b]{0.45\textwidth}{\centering \includegraphics[width=0.45\textwidth]{graphics/Zweiparabeln_Set_xi_0_0_eps_00_00_SD6-eps-converted-to.pdf}\\(c)}
%	\parbox[b]{0.45\textwidth}{\centering \includegraphics[width=0.45\textwidth]{graphics/Zweiparabeln_Set_xi_0_0_eps_01_01_SD6-eps-converted-to.pdf}\\(d)} \\
	\parbox[b]{0.45\textwidth}{\centering \includegraphics[width=0.45\textwidth]{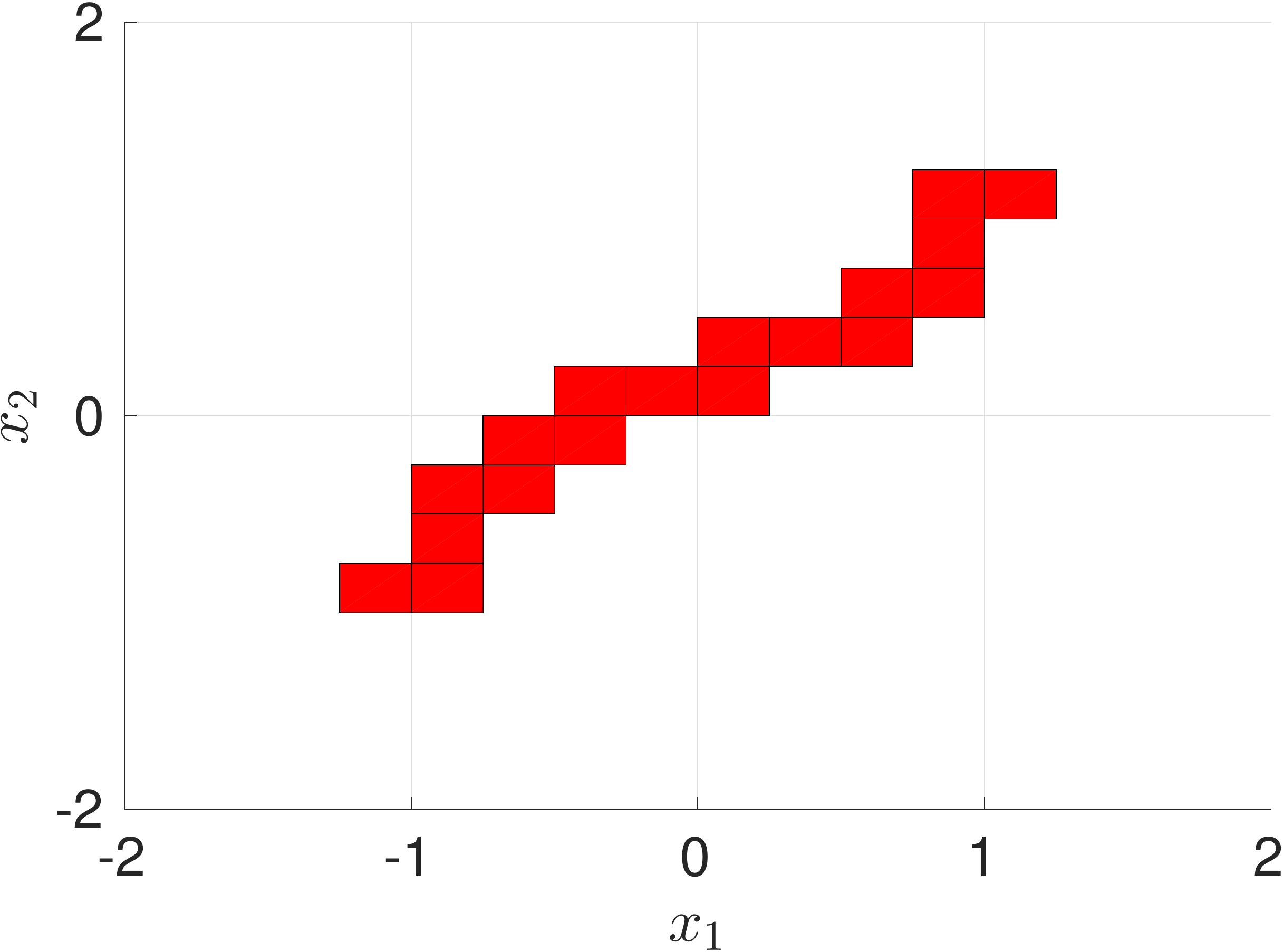}\\(c)}\hspace{0.06\textwidth}
	\parbox[b]{0.45\textwidth}{\centering \includegraphics[width=0.45\textwidth]{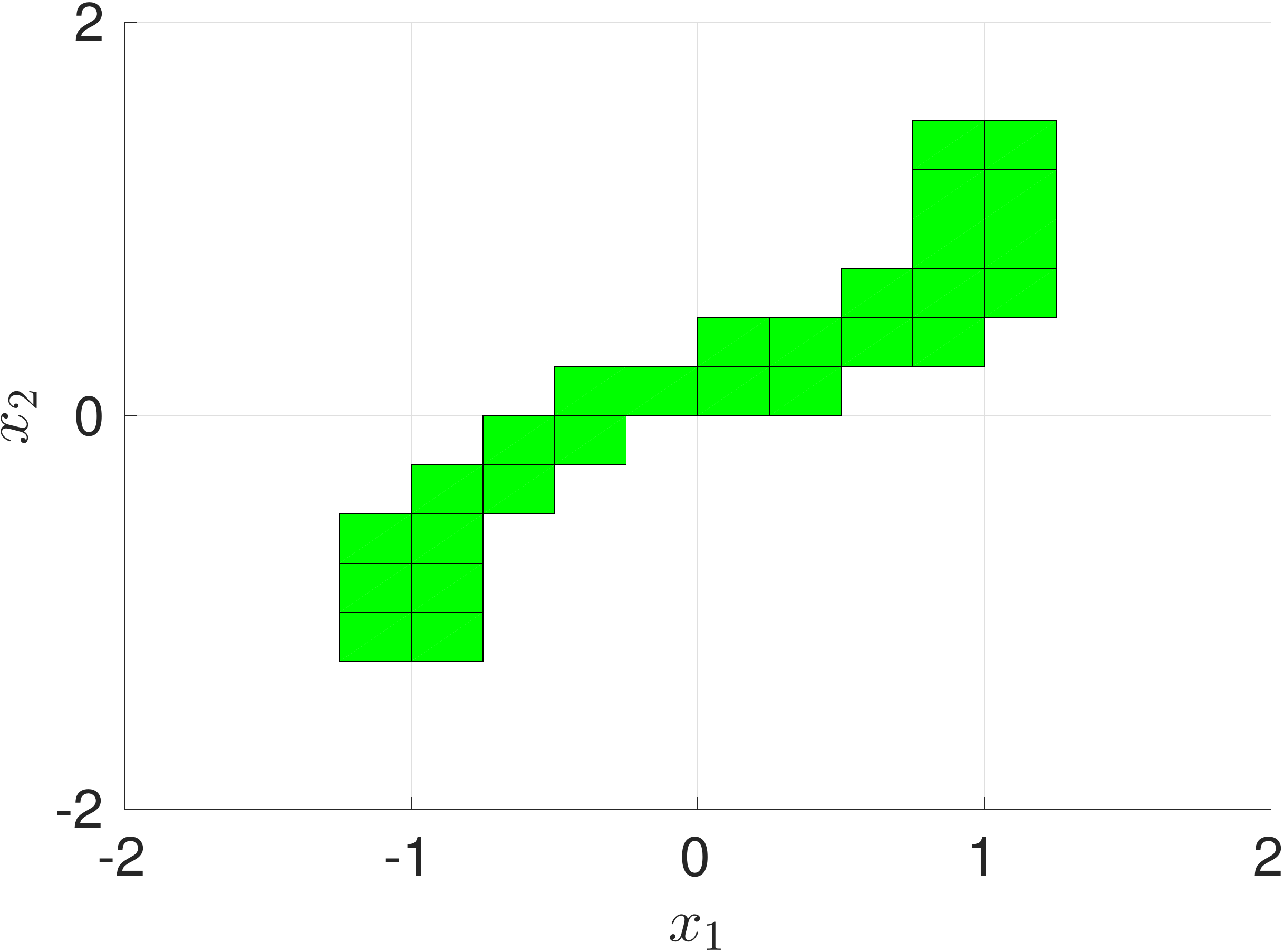}\\(d)} \\
%	\parbox[b]{0.45\textwidth}{\centering \includegraphics[width=0.45\textwidth]{graphics/Zweiparabeln_Set_xi_0_0_eps_00_00_SD10-eps-converted-to.pdf}\\(g)}
%	\parbox[b]{0.45\textwidth}{\centering \includegraphics[width=0.45\textwidth]{graphics/Zweiparabeln_Set_xi_0_0_eps_01_01_SD10-eps-converted-to.pdf}\\(h)} \\
	\parbox[b]{0.45\textwidth}{\centering \includegraphics[width=0.45\textwidth]{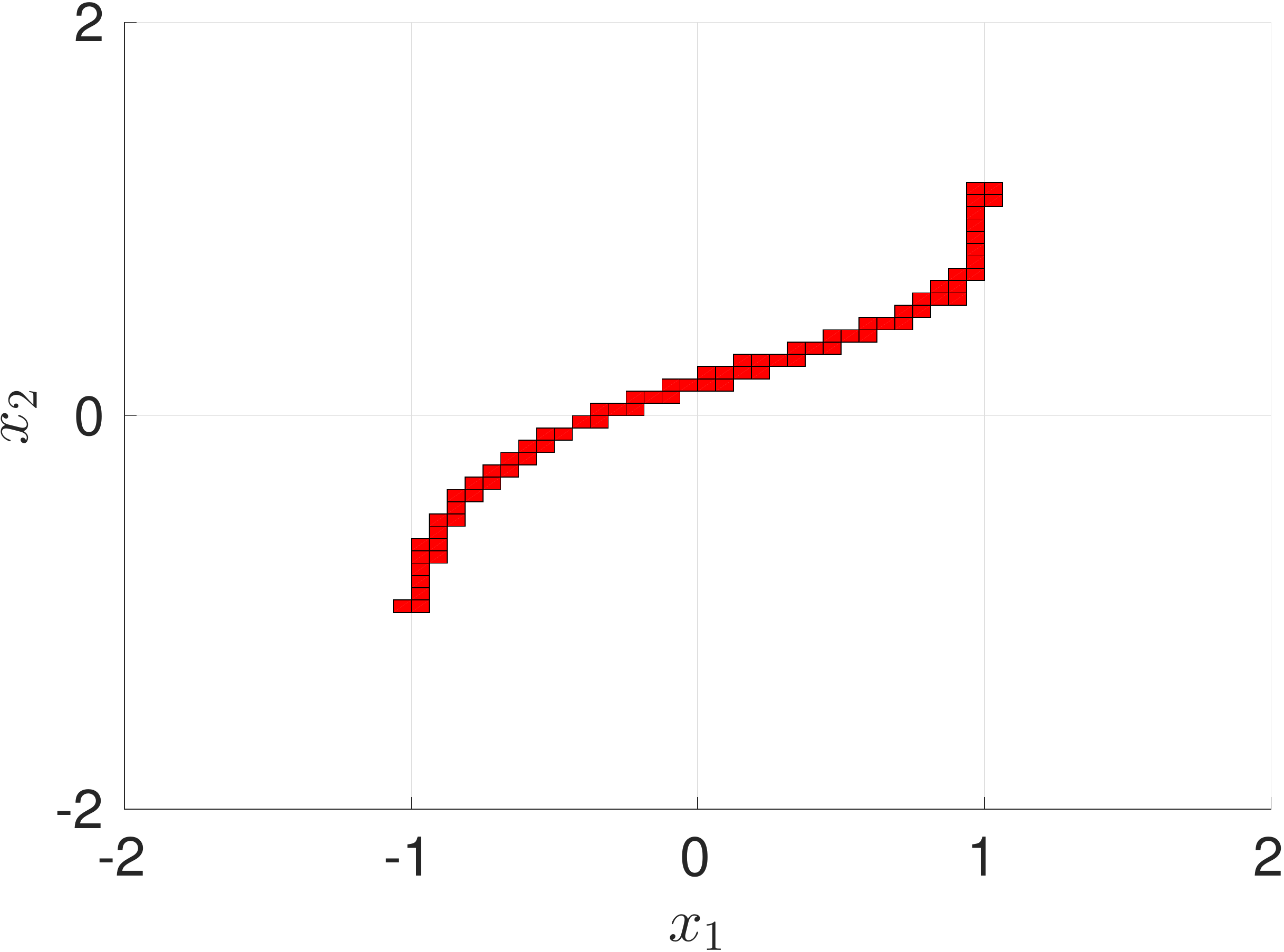}\\(e)}\hspace{0.06\textwidth}
	\parbox[b]{0.45\textwidth}{\centering \includegraphics[width=0.45\textwidth]{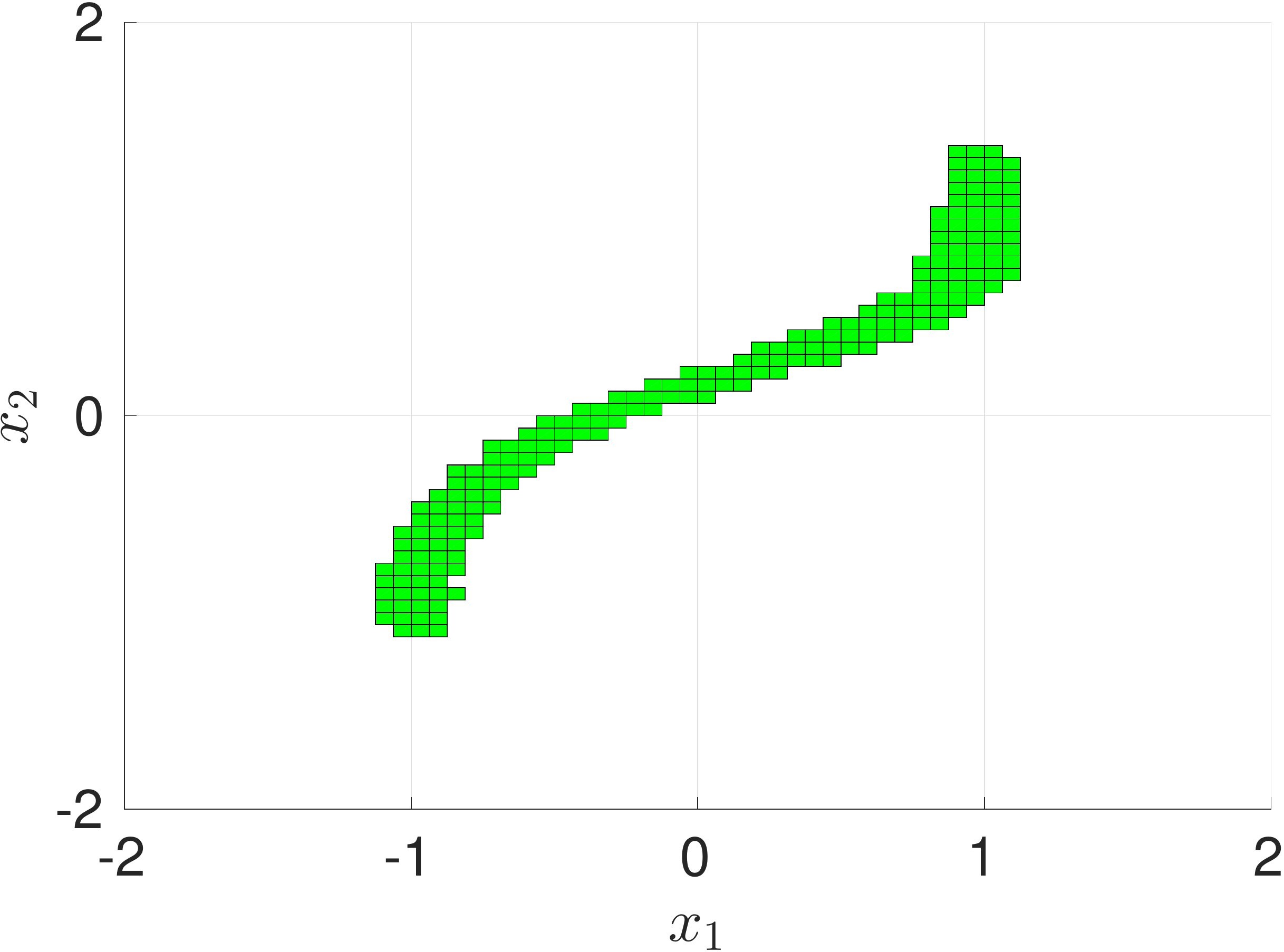}\\(f)} \\
	\caption{Comparison of the inexact and the exact solution of problem~\eqref{eq:example_zweiparabeln} after 4, 8 and 12 subdivision steps.}
	\label{fig:zweiparabeln_Stages}
\end{figure}

In the numerical realization, we approximate each box by an equidistant grid with two points in each direction, i.e.~by four sample points in total. This results in a total number of $\approx 50,000$ function evaluations for the exact problem. 
The number of boxes is much higher for the inexact solution, in this case by a factor of $\approx 8$. This is not surprising since the equality condition~\eqref{eq:MOP_optimality} is now replaced by the inequality condition~\eqref{eq:PS_epsilon}. Hence, the approximated set is no longer a $k-1$-dimensional object. All sets $B \in \mathcal{B}_s$ that satisfy the inequality condition~\eqref{eq:PS_epsilon} are not discarded in the selection step. Consequently, at a certain box size, the number of boxes increases exponentially with decreasing diameter diam$(B)$ (cf.~Figure~\ref{fig:Hausdorff}~(b)). The result is an increased computational effort for later iterations. This is visualized in Figure~\ref{fig:zweiparabeln_Stages}, where the solutions at different stages of the subdivision algorithm are compared. A significant difference between the solutions can only be observed in later stages (Figure~\ref{fig:zweiparabeln_Stages}~(e) and (f)). For this reason, an adaptive strategy needs to be developed where boxes satisfying~\eqref{eq:PS_epsilon} remain within the box collection but are no longer considered in the subdivision algorithm.

As a second example, we consider the function $F:\R^3 \rightarrow \R^3$:
\begin{align}
	\min_{x\in\R^3} F(x) = \min_{x\in\R^3} \left( \begin{array}{c}
	(x_1 - 1)^4 + (x_2 - 1)^2 + (x_3 - 1)^2 \\
	(x_1 + 1)^2 + (x_2 + 1)^4 + (x_3 + 1)^2 \\
	(x_1 - 1)^2 + (x_2 + 1)^2 + (x_3 - 1)^4
	\end{array} \right). \label{eq:example_dreiparabeln}
\end{align}
The observed behavior is very similar to the two-dimensional case, cf.~Figure~\ref{fig:dreiparabeln}, where in (b) the box covering for the inexact problem is shown as well as the iso-surface $\|q(x)\|_2 = 2 \|\epsilon\|_{\infty}$. One can see that the box covering lies inside this iso-surface except for small parts of some boxes. This is due to the finite box size and the fact that at least one sample point is mapped into the box itself. For smaller box radii, this artifact does no longer occur.
\begin{figure}
	\centering
	\parbox[b]{0.49\textwidth}{\centering \includegraphics[width=0.45\textwidth]{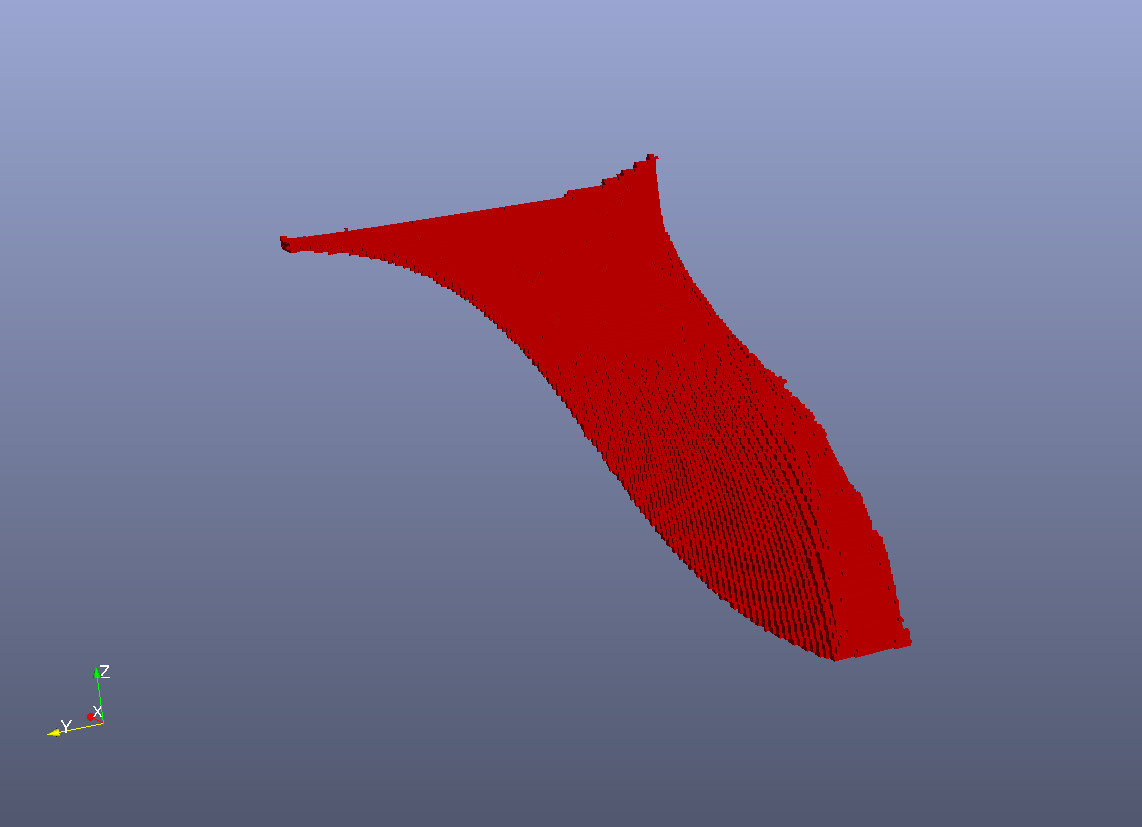}\\(a)}
	\parbox[b]{0.49\textwidth}{\centering \includegraphics[width=0.45\textwidth]{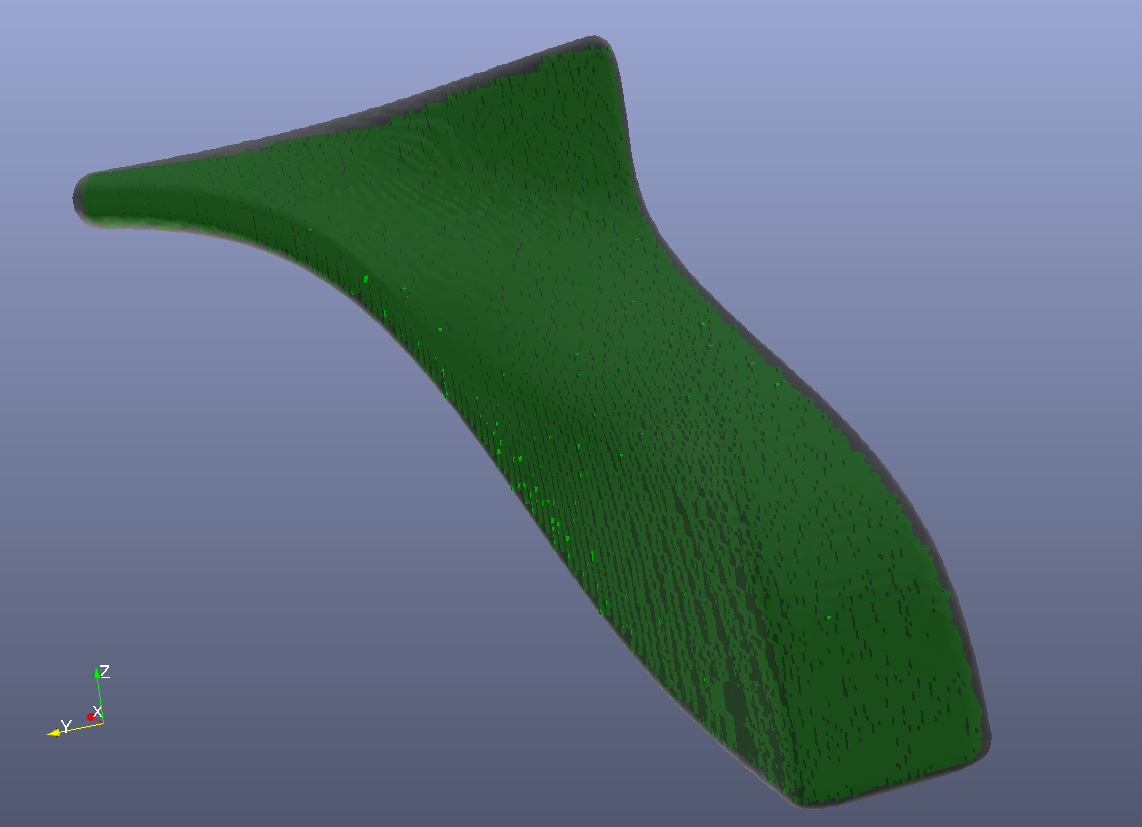}\\(b)} \\ \vspace{0.25cm}
	\parbox[b]{0.49\textwidth}{\centering \includegraphics[width=0.45\textwidth]{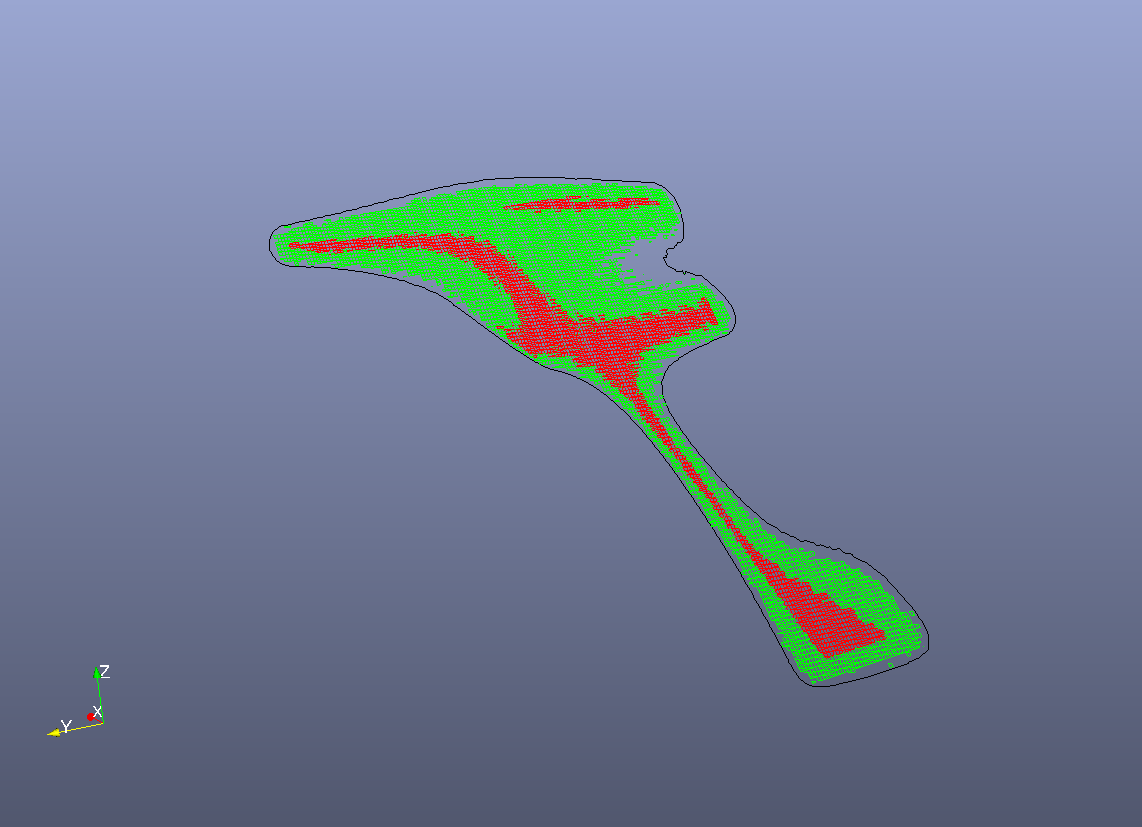}\\(c)}
	\parbox[b]{0.49\textwidth}{\centering \includegraphics[width=0.45\textwidth]{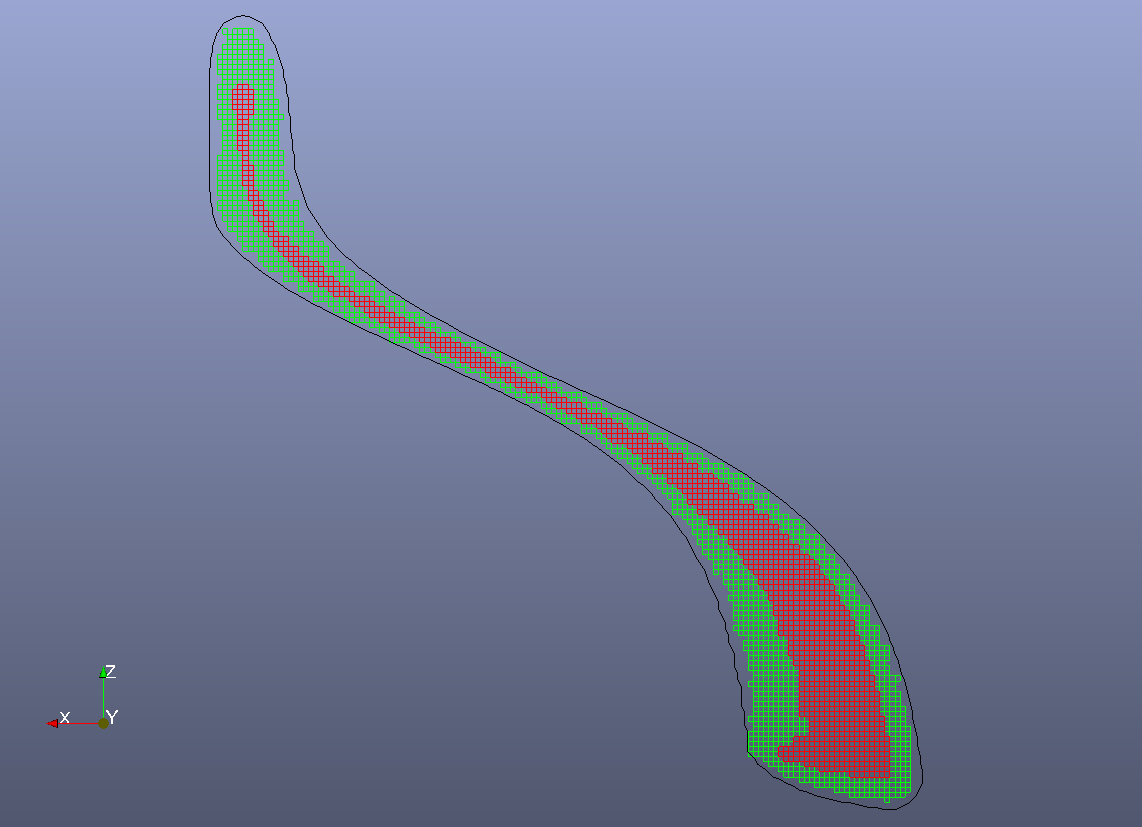}\\(d)} \\
	\parbox[b]{0.9\textwidth}{\centering \includegraphics[width=0.9\textwidth]{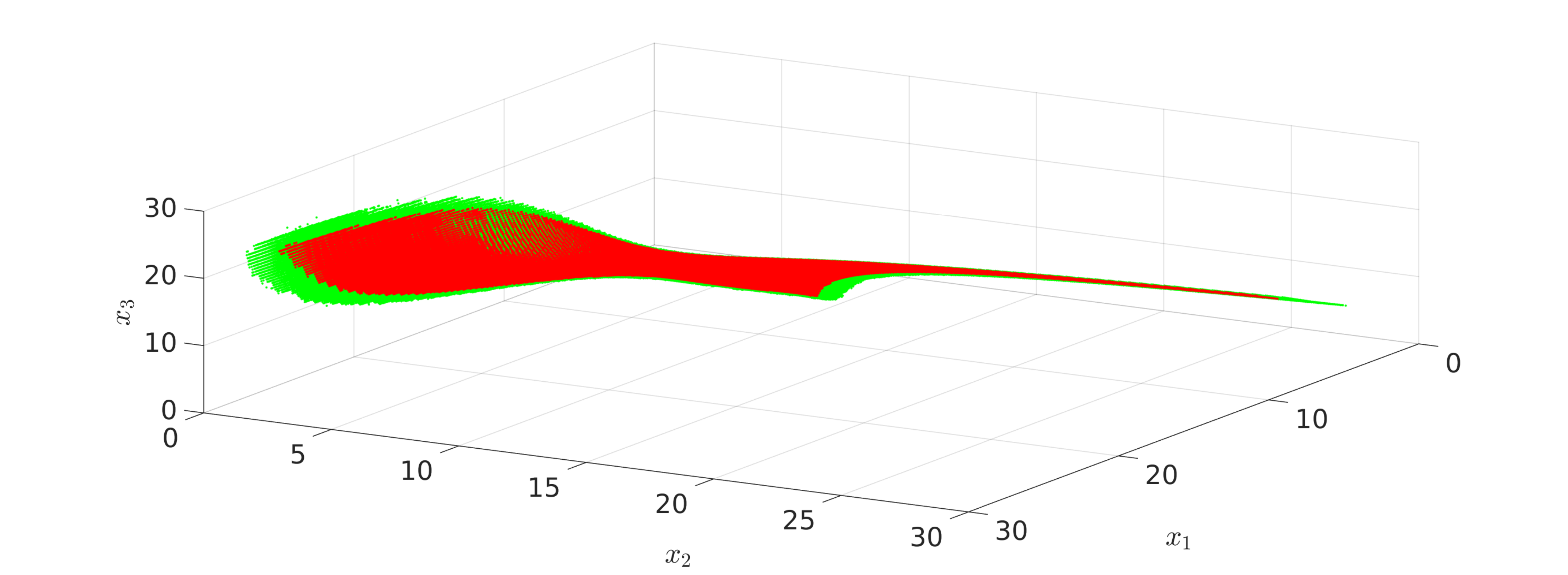}\\(e)}
	\caption{Box coverings of the Pareto set (exact solution: red, inexact solution: green) of problem~\eqref{eq:example_dreiparabeln} after $24$ subdivision steps (diam$(\mathcal{B}_0) = 4$, diam$(\mathcal{B}_{24}) = 1/2^{6}$). (a) Solution without errors. (b) Solution with $\xi = (0,0,0)^{\top}$ and $\epsilon_i = (0.1,0.1,0.1)^{\top}$. The iso-surface ($\|q(x)\|_2 = 2 \|\epsilon\|_{\infty} = 0.2$) is the upper bound of the error. (c)--(d) Two-dimensional cut planes through the Pareto set. The point of view in (c) is as in (a), (b) and (d) is a cut plane parallel to the $x_2$ plane at $x_2=-0.9$. (e) The corresponding Pareto fronts.}
	\label{fig:dreiparabeln}
\end{figure}

Finally, we consider an example where the Pareto set is disconnected. This is an example from production and was introduced in \cite{SSW02}. We want to minimize the failure of a product which consists of $n$ components. The probability of failing is modeled individually for each component and depends on the additional cost $x$:
\begin{align*}
	p_1(x) &= 0.01 \exp\left(-(x_1/20)^{2.5}\right), \\
	p_2(x) &= 0.01 \exp\left(-(x_2/20)^{2.5}\right), \\
	p_j(x) &= 0.01 \exp\left(-x_1/15\right), \qquad j = 3, \ldots, n.
\end{align*}
The resulting MOP is thus to minimize the failure and the additional cost at the same time:
\begin{align}
	\min_{x\in\R^n} F(x) = \min_{x\in\R^n} \left( \begin{array}{c}
		\sum_{j=1}^{n} x_j \\
		1 - \sum_{j=1}^{n} \left( 1 - p_j(x) \right)
	\end{array} \right). \label{eq:example_SSW_MOP}
\end{align}
We now assume that the additional cost $x$ is subject so some uncertainty, e.g.~due to varying prices, and set $| \widetilde{x}_i - x_i | < 0.01$ for $i = 1, \ldots, n$. Using this, we can estimate the error bounds within the initial box $\mathcal{B}_0 = [0,40]^{n}$ and obtain $\xi = (0.05, 2 \cdot 10^{-5})^\top$ and $\epsilon = (0, 8 \cdot 10^{-7})^\top$.
\begin{figure}
	\centering
	\parbox[b]{0.47\textwidth}{\centering \includegraphics[width=0.47\textwidth]{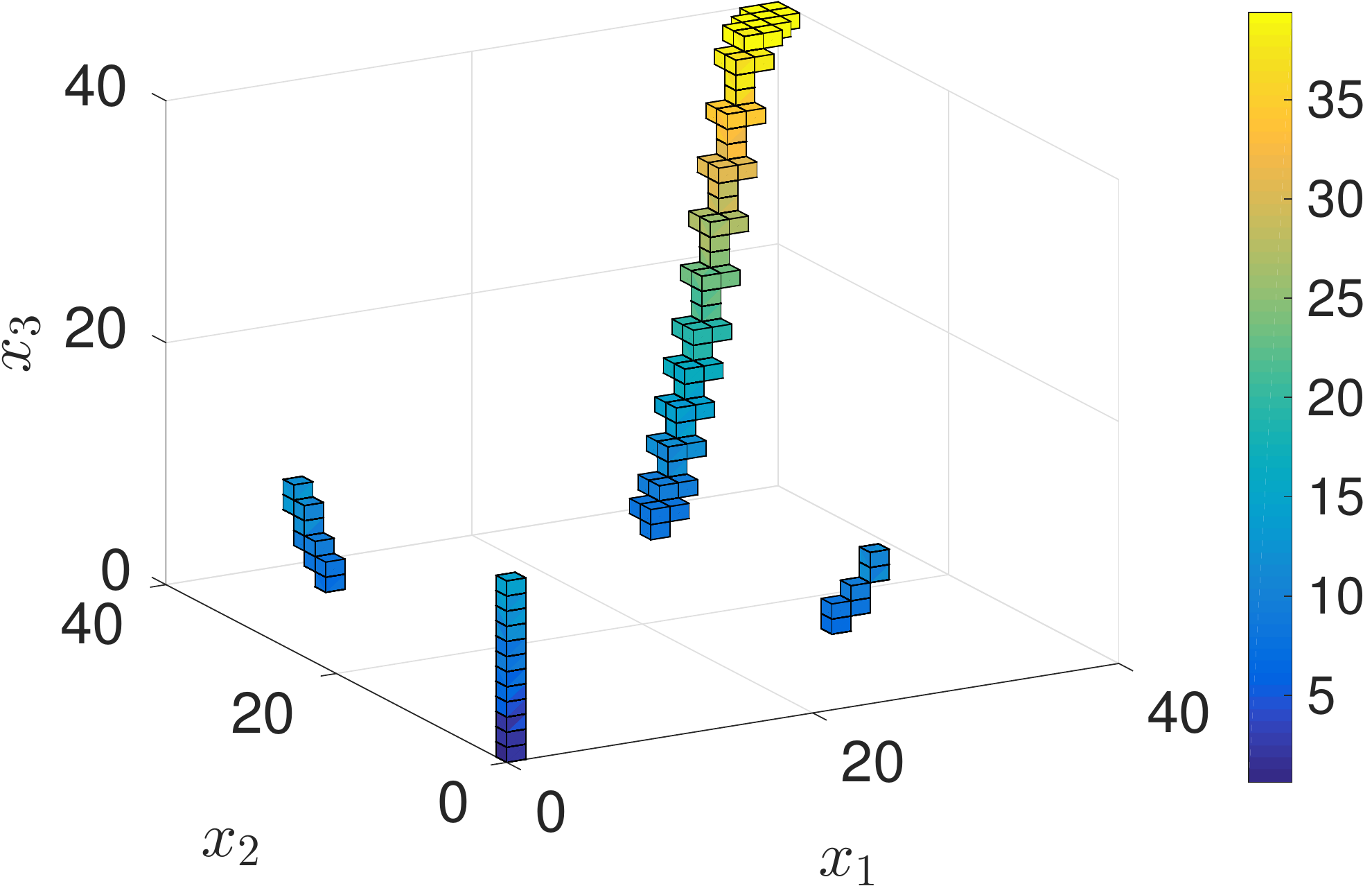}\\(a)} \quad
	\parbox[b]{0.47\textwidth}{\centering \includegraphics[width=0.47\textwidth]{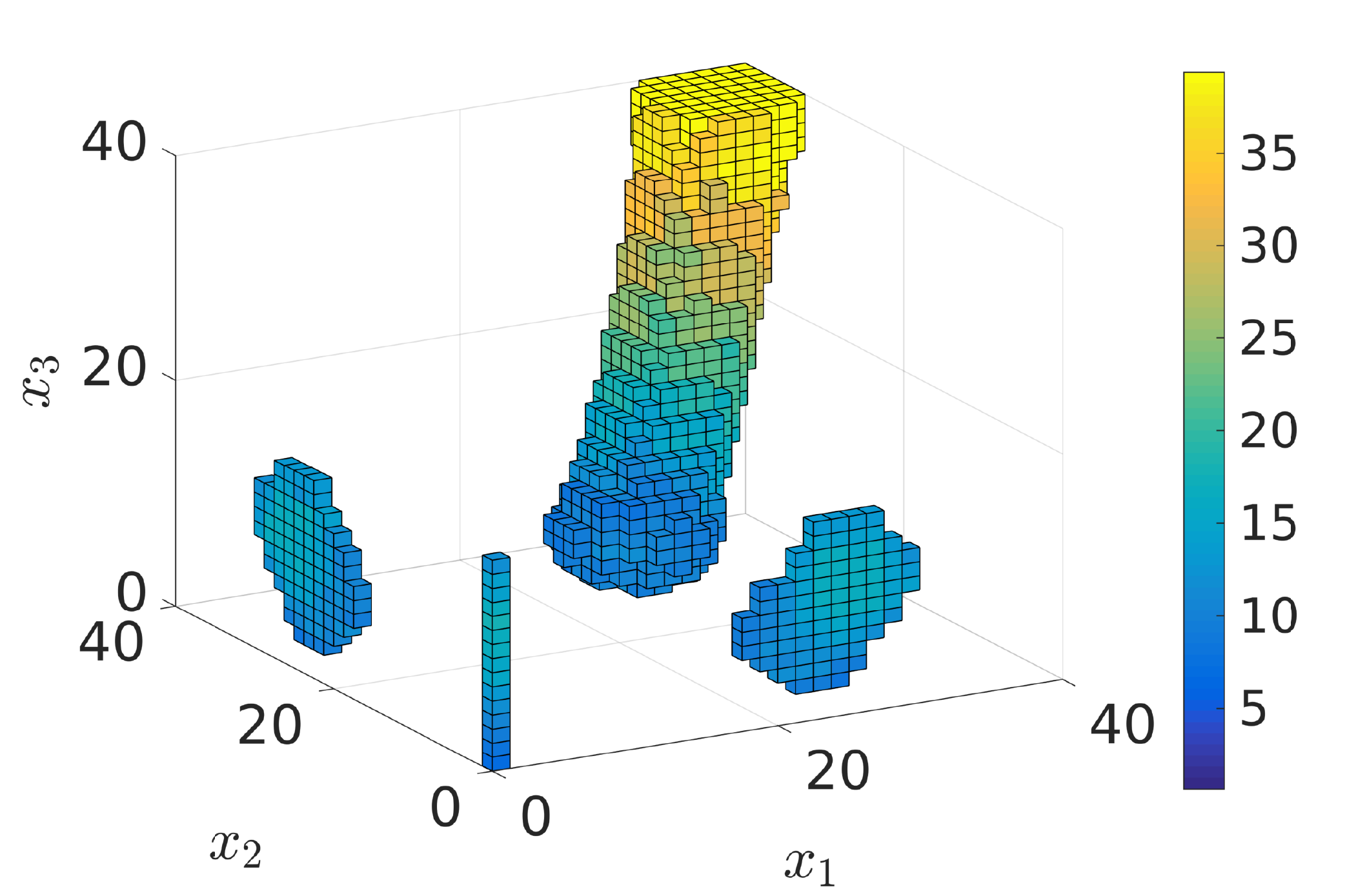}\\(b)} \\~\\
	\parbox[b]{0.47\textwidth}{\centering \includegraphics[width=0.45\textwidth]{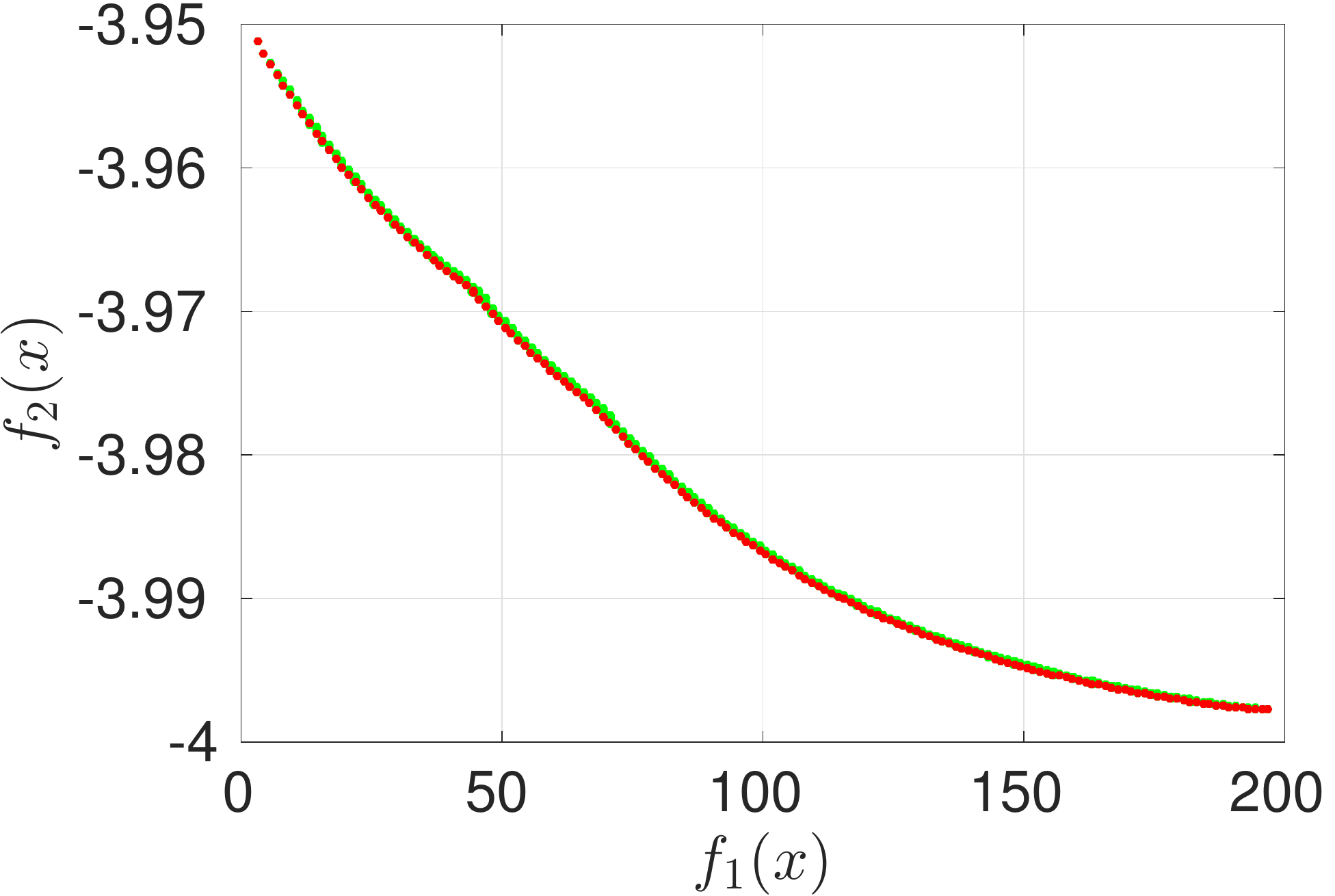}\\(c)} \qquad
	\parbox[b]{0.47\textwidth}{\centering \includegraphics[width=0.45\textwidth]{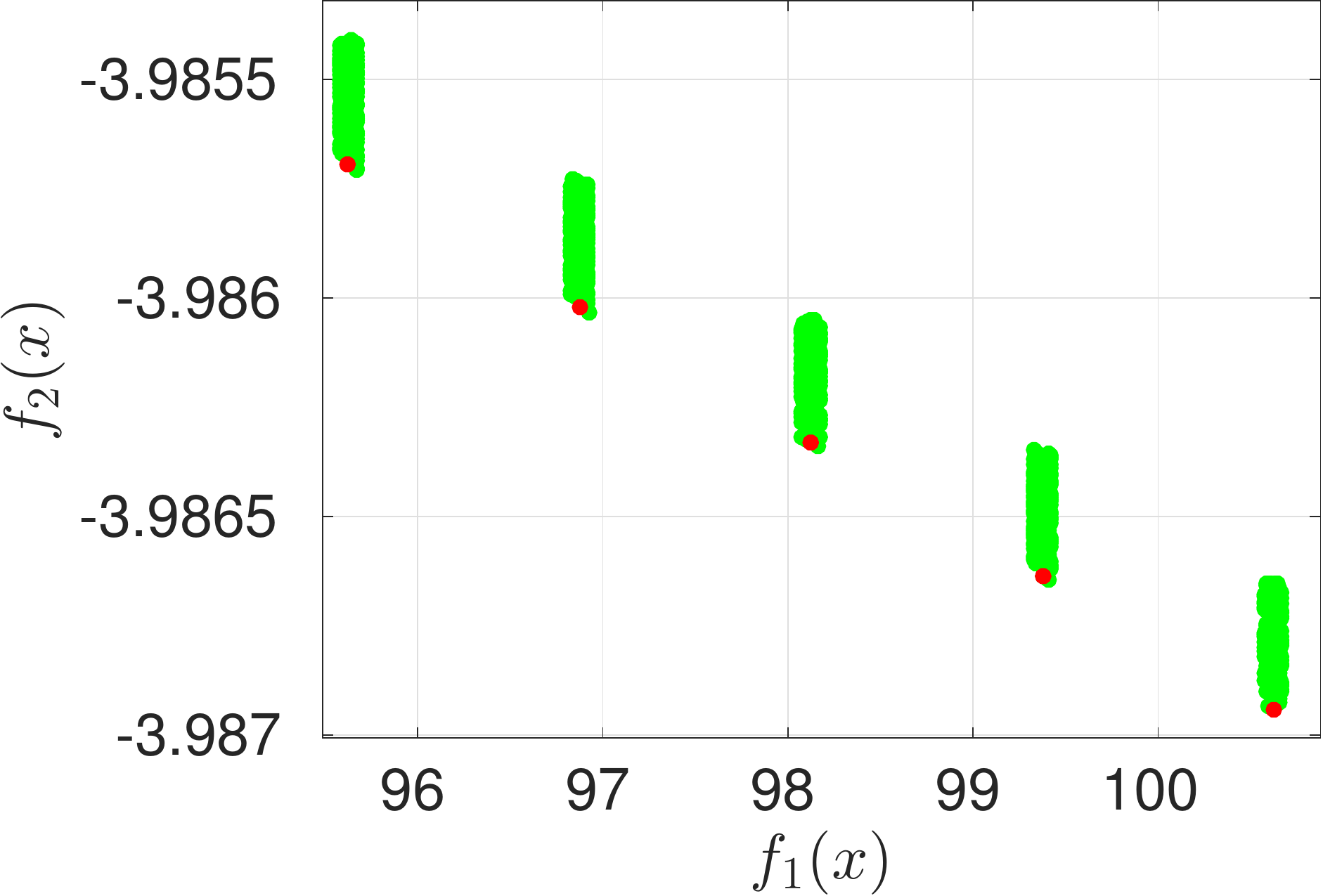}\\(d)}
	\caption{(a) Projection of the box covering of $\PS$ for problem~\eqref{eq:example_SSW_MOP} with $n=5$ after $25$ subdivision steps (diam$(\mathcal{B}_0) = 40$, diam$(\mathcal{B}_{25}) = 1.25$). The coloring represents the fourth component $x_4$. (b) Box covering of $\PSxi$ with inexact data $\widetilde{x}$ with $| \widetilde{x}_i - x_i | < 0.01$ for $i = 1, \ldots, 5$. (c)--(d) The Pareto fronts corresponding to (a) and (b) in green and red, respectively. The points are the images of the box centers.}
\end{figure}

Since the set of substationary points is disconnected, we here utilize the Sampling Algorithm~\ref{algo:Sampling}. For $n=5$, the resulting Pareto sets $\PS$ and $\PSxi$ are depicted in Figures~\ref{fig:production}~(a) and (b), respectively. Due to the small gradient of the objectives, his results in a significantly increased number of boxes by a factor of $\approx 300$ (cf.~Figure~\ref{fig:Hausdorff}~(b)). 

The quality of the solution can be measured using the Hausdorff distance $d_h(\PS,\PSxi)$ \cite{RW98}. This is depicted in Figure~\ref{fig:Hausdorff}~(a), where the distance between the exact and the inexact solution is shown for all subdivision steps for the three examples above. We see that the distance reaches an almost constant value in the later stages. This distance is directly influenced by the upper bounds $\epsilon$ and $\xi$, respectively. However, it cannot simply be controlled by introducing a bound on the error in the objectives or gradients since it obviously depends on the objective functions. Hence, in order to limit the error in the decision space, further assumptions on the objectives have to be made.
\begin{figure}
	\centering
	\parbox[b]{0.47\textwidth}{\centering \includegraphics[width=0.4\textwidth]{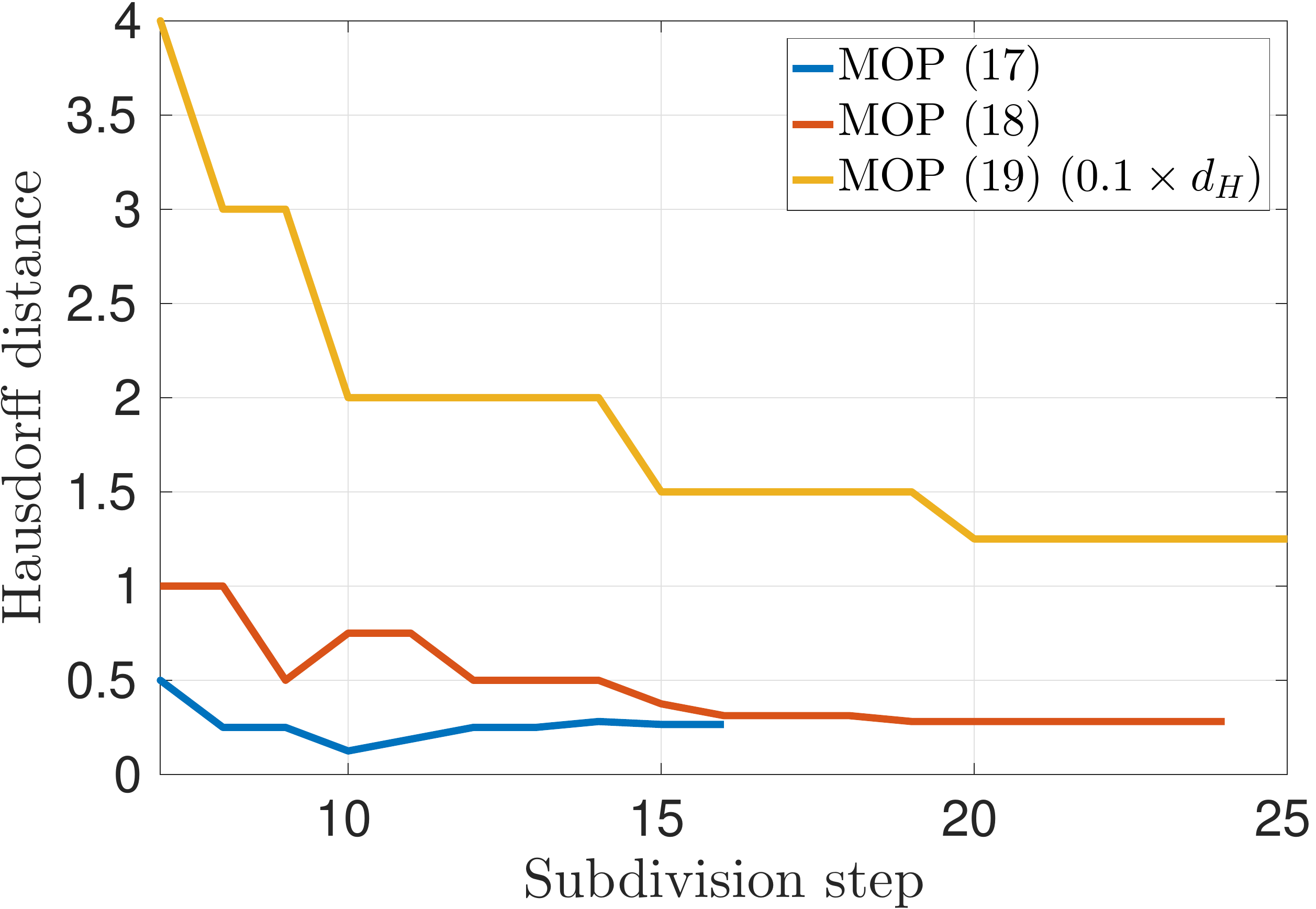}\\(a)} \quad
	\parbox[b]{0.47\textwidth}{\centering \includegraphics[width=0.39\textwidth]{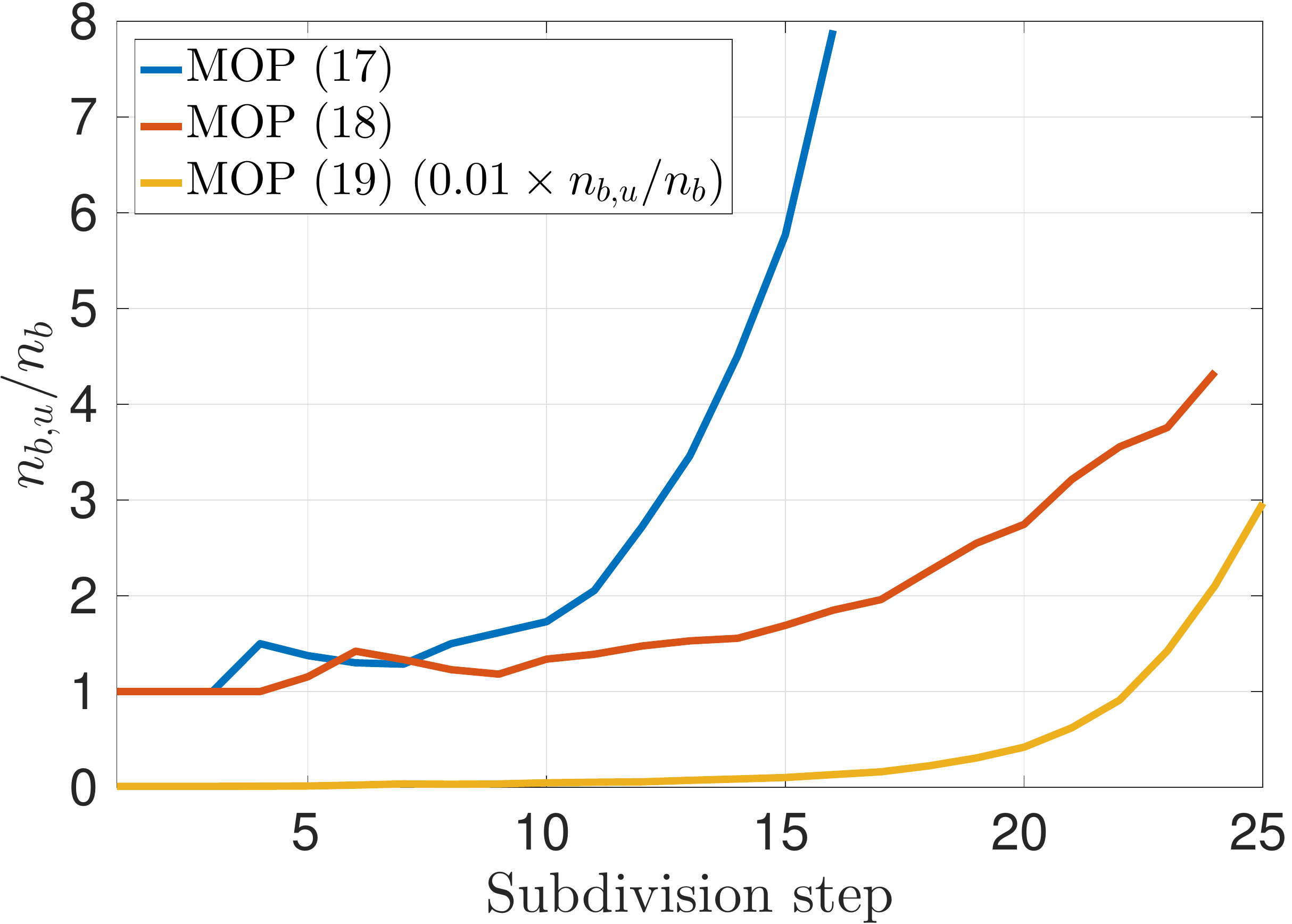}\\(b)}
	\caption{(a) Hausdorff distance $d_h(\PS,\PSxi)$ for the three examples above. (b) The corresponding ratio of box numbers between the inexact solution ($n_{b,u}$) and the exact solution ($n_b$).}
	\label{fig:Hausdorff}
\end{figure}

%%%%%%%%%%%%%%%%%%%%%%%%%%%%%%%%%%%%%%%%%%%%%%%%%%%%%%%%%%%%%%%%%%%%%%%%%%%%%%%%%%%%%%%%%%%%%%
%% Conclusion
%%%%%%%%%%%%%%%%%%%%%%%%%%%%%%%%%%%%%%%%%%%%%%%%%%%%%%%%%%%%%%%%%%%%%%%%%%%%%%%%%%%%%%%%%%%%%%
\section{Conclusion}
\label{sec:Conclusion}
In this article, we present an extension to the subdivision algorithms developed in \cite{DSH05} for MOPs with uncertainties in the form of inexact function and gradient information. An additional condition for a descent direction is derived in order to account for the inaccuracies in the gradients. Convergence of the extended subdivision algorithm to a superset of the Pareto set is proved and an upper bound for the maximal distance to the set of substationary points is given. When taking into account errors, the number of boxes in the covering of the Pareto set increases, especially for later iterations, causing larger computational effort. For this reason, an adaptive strategy needs to be developed where boxes approximately satisfying the KKT conditions (according to \eqref{eq:PS_epsilon}) remain within the box collection but are no longer considered in the subdivision algorithm. Furthermore, we intend to extend the approach to constrained MOPs in the future. A comparison to other methods, especially memetic algorithms, would be interesting to investigate the numerical efficiency of the presented method.
Finally, we intend to combine the subdivision algorithm with model order reduction techniques in order to solve multiobjective optimal control problems constrained by PDEs (see e.g.~\cite{BBV16} for the solution of bicriterial MOPs with the \emph{reference point} approach or \cite{POBD15} for multiobjective optimal control of the Navier-Stokes equations). For the development of reliable and robust multiobjective optimization algorithms, error estimates for reduced order models \cite{TV09} need to be combined with concepts for inexact function values and gradients. The idea is to set upper bounds for the errors $\xi$ and $\epsilon$ and thereby define the corresponding accuracy requirements for the reduced order model.
\\~

\textbf{Acknowledgement:} This work is supported by the Priority Programme SPP 1962 ``Non-smooth and Complementarity-based Distributed Parameter Systems'' of the German Research Foundation (DFG).

%%%%%%%%%%%%%%%%%%%%%%%%%%%%%%%%%%%%%%%%%%%%%%%%%%%%%%%%%%%%%%%%%%%%%%%%%%%%%%%%%%%%%%%%%%%%%%
%% Bibliography
%%%%%%%%%%%%%%%%%%%%%%%%%%%%%%%%%%%%%%%%%%%%%%%%%%%%%%%%%%%%%%%%%%%%%%%%%%%%%%%%%%%%%%%%%%%%%%
\newpage
\bibliographystyle{alphaabbr}
\bibliography{Bib}
\end{document}